\documentclass{article}
\usepackage{amssymb, amsmath, amsthm, a4wide, hyperref}
\newcounter{counter}

\newtheorem{lemma}[counter]{Lemma}
\newtheorem{theorem}[counter]{Theorem}
\newtheorem{remark}[counter]{Remark}

\newtheorem{definition}[counter]{Definition}

\newcommand{\bzeta}{\boldsymbol{\zeta}}
\begin{document}
\title{$p$-adic equidistribution and an application to $S$-units}
\author{Gerold Schefer}
\maketitle

\paragraph{Abstract:}We prove a Galois equidistribution result for torsion points in $\mathbb G_m^n$ in the $p$-adic setting for test functions of the form $\log |F|_p$ where $F$ is a nonzero polynomial with coefficients in the $p$-adic numbers. Our result includes a power saving quantitative estimate of the decay rate rate of the equidistribution. As an application we show that Ih's Conjecture is true for a class of divisors of $\mathbb G_m^n$.

\section{Introduction}
Let $n\ge 1$ be an integer and let $\mathbb G_m^n$ denote the algebraic torus with some algebraically closed base field $K$. We will identify $\mathbb G_m^n$ with the group $(K\setminus \{0\})^n$. For a point $\boldsymbol\zeta\in \mathbb G_m^n$ of finite order we define
$$\delta(\boldsymbol\zeta) = \inf\{|\mathbf a|: \mathbf a\in\mathbb Z^n\setminus \{0\} \text{ with }\boldsymbol\zeta^{\mathbf a}=1\}$$ which measures the multiplicative dependence of the entries of $\boldsymbol\zeta$; the notation $\boldsymbol\zeta^{\mathbf a}$ is given in Section \ref{sec:notation} and $|\cdot|$ denotes the maximum norm.

Let us assume for the moment that $K=\mathbb C$. For any $P\in\mathbb C[X_1^{\pm 1},\dots, X_n^{\pm 1}]\setminus \{0\}$ the integral $$m(P)=\int_{[0,1)^d}\log|P(\mathbf {e(x)})| d\mathbf x$$ exists and is called the logarithmic Mahler measure of $P$. Here $\mathbf {e(x)}=(e^{2\pi i x_1},\dots, e^{2\pi i x_n})$ for any ${\mathbf x=(x_1,\dots,x_n)\in \mathbb R^n}$. It makes even sense for $\mathbf x\in (\mathbb R/\mathbb Z)^n$.
A torsion coset of $\mathbb G_m^n$ is the translate of a connected algebraic subgroup of $\mathbb G_m^n$ by a point of finite order. We call a torsion coset proper, if it does not equal $\mathbb G_m^n$. Let $S^1=\{z\in\mathbb C: |z|=1\}$ denote the unit circle. We call ${P\in\mathbb C[X_1^{\pm 1},\dots, X_n^{\pm 1}]}$ essentially atoral if the Zariski closure of $$\{(z_1,\dots,z_n)\in (S^1)^n: P(z_1,\dots, z_n)=0\}$$ in $\mathbb G_m^n$ is a finite union of irreducible algebraic sets of codimension at least $2$ and proper torsion cosets. Let $\overline{\mathbb Q}$ denote the algebraic closure of $\mathbb Q$ in $\mathbb C$. Let $P\in\overline{\mathbb Q}[X_1,\dots, X_n]\setminus \{0\}$. By the Manin-Mumford Conjecture for $\mathbb G_m^n$ which was proven by Laurent in \cite{Laurent}, the Zariski closure of the torsion points in $V=\{\mathbf x\in\mathbb G_m^n: P(\mathbf x)=0\}$ is a finite union of torsion cosets. Let $\boldsymbol\zeta\in\mathbb G_m^n$ be a torsion point of order $N$ and $H<\mathbb G_m^n$ a connected algebraic subgroup. We claim that there is a constant $C>0$ such that for all torsion points $\boldsymbol\xi\in \boldsymbol\zeta H$ we have $\delta(\boldsymbol\xi)\le C$. By Corollary 3.2.15 in \cite{BombGub} there is a non trivial subgroup $\Lambda<\mathbb Z^n$ such that $H=\{ \mathbf x\in\mathbb G_m^n: \mathbf x^{\boldsymbol\lambda} = 1\text{ for all }\boldsymbol\lambda\in\Lambda\}$. Let $\boldsymbol \xi=\boldsymbol \zeta\mathbf h\in \boldsymbol\zeta H$ be a torsion point and $\boldsymbol\lambda\in\Lambda\setminus\{0\}$. Then we have $$\boldsymbol\xi^{N\boldsymbol\lambda}=\left(\boldsymbol\zeta^N\right)^{\boldsymbol\lambda}\left(\boldsymbol h^{\boldsymbol\lambda}\right)^N=1 \text{ and therefore }\delta(\boldsymbol\xi)\le N|\boldsymbol\lambda|.$$ This implies that there is a constant $D>0$ such that for every torsion point $\boldsymbol \xi\in V$ we have $\delta(\boldsymbol\xi)\le D$. So in particular we have $P(\boldsymbol\xi)\ne 0$ if $\delta(\boldsymbol\xi)>D$. The main result in \cite{Atoral} is
\begin{theorem}\label{thm:atoral}For each essentially atoral $P\in \overline{\mathbb Q}[X_1^{\pm1},\dots,X_n^{\pm 1}]\setminus \{0\}$ there exists $\kappa>0$ with the following property. Suppose $\boldsymbol\zeta\in\mathbb G_m^n(\mathbb C)$ has finite order and $\delta(\boldsymbol\zeta)$ is sufficiently large. Then $P(\sigma(\boldsymbol\zeta))\ne 0$ for all $\sigma\in\mathrm{Gal}(\mathbb Q(\boldsymbol\zeta)/\mathbb Q)$ and $$\frac{1}{[\mathbb Q(\boldsymbol\zeta):\mathbb Q]}\sum_{\sigma\in\mathrm{Gal}(\mathbb Q(\boldsymbol\zeta)/\mathbb Q)} \log |P(\sigma(\boldsymbol\zeta))| = m(P)+O(\delta(\boldsymbol\zeta)^{-\kappa})$$ as $\delta(\boldsymbol\zeta)\to\infty$, where the implicit constant depends only on $n$ and $P$.\end{theorem}
The aim of this note is to prove a $p$-adic analogue of Theorem \ref{thm:atoral} for a prime $p$. We work over $K=\mathbb C_p$, a fixed completion of $\overline{\mathbb Q}_p,$ where $\mathbb Q_p$ is the field of the $p$-adic numbers. Let ${F\in \mathbb C_p[X_1^{\pm 1},\dots, X_n^{\pm 1}]}$. The analogue of the Mahler measure is $\log|F|_p$, where $|F|_p$ is the maximum of the $p$-adic absolute values of its coefficients also known as the $p$-adic Gauss norm. Proposition 2.2 in \cite{Ih} can be formulated as follows.
\begin{theorem}
	Let $K\subseteq {\mathbb C_p}$ be a number field and $F\in K[X_1,\dots, X_n]\setminus\{0\}$. Let $(\bzeta_k)_{k\in\mathbb N}$ be a sequence of points of finite order in $\mathbb G_m^n(\mathbb C_p)$ such that $\lim_{k\to\infty} \delta(\bzeta_k) = \infty$ and such  that $F(\sigma(\bzeta_k))\ne 0$ for all $\sigma\in \mathrm{Gal}(\mathbb Q(\bzeta_k)/\mathbb Q(\bzeta_k)\cap K)$. Then we have
	$$\lim_{k\to\infty}\frac{1}{[\mathbb{Q}(\bzeta_k):\mathbb Q(\bzeta_k)\cap K]}\sum_{\sigma\in\mathrm{Gal}(\mathbb{Q}(\bzeta_k):\mathbb Q(\bzeta_k)\cap K)}\log|F(\sigma(\bzeta_k))|_p = \log|F|_p.$$
\end{theorem} Using the same strategy as in \cite{Atoral} we can improve this result. We do not need the algebraicity of the coefficients and we are able to give a bound for the error. Indeed we can prove

\begin{theorem}\label{thm:pSimple}For every $F\in \mathbb C_p[X_1^{\pm1},\dots,X_n^{\pm 1}]\setminus \{0\}$ there exists $\kappa>0$ with the following property. Suppose $\boldsymbol\zeta\in\mathbb G_m^n$ has finite order and $\delta(\boldsymbol\zeta)$ is sufficiently large in terms of $F$. Then $F(\sigma(\boldsymbol\zeta))\ne 0$ for all $\sigma\in\mathrm{Gal}(\mathbb Q(\boldsymbol\zeta)/\mathbb Q)$ and $$\frac{1}{[\mathbb Q(\boldsymbol\zeta):\mathbb Q]}\sum_{\sigma\in\mathrm{Gal}(\mathbb Q(\boldsymbol\zeta)/\mathbb Q)} \log |F(\sigma(\boldsymbol\zeta))|_p = \log |F|_p+O(\delta(\boldsymbol\zeta)^{-\kappa})$$ as $\delta(\boldsymbol\zeta)\to\infty$, where the implicit constant depends only on $n$ and $F$.\end{theorem}

A more general theorem is given in Theorem \ref{thm:padic}. Note that no analogue of the hypothesis being essentially atoral is needed. The reason is the following theorem of Tate and Voloch, Theorem 2 in \cite{TateVoloch}

\begin{theorem}\label{thm:TV}For every integer $n\ge 1$ and every family $a_1,\dots, a_n\in\mathbb C_p$, there exists a constant $c>0$ such that, for any $\zeta_1,\dots, \zeta_n$ roots of unity in $\mathbb C_p$, either $\sum_{i=1}^n \zeta_i a_i=0$ or $|\sum_{i=1}^n \zeta_i a_i|_p \ge c.$\end{theorem}

Unlike in the complex case, where linear forms in roots of unity can get as close to zero as we wish, we have a positive lower bound when working in $\mathbb C_p$. So we can allow all polynomials and do not have to require additional conditions as in Theorem \ref{thm:atoral}, where $P$ must be essentially atoral. The field where the coefficients of the polynomial are from, is as general as it can be.\\

As an application we prove a special case of Ih's Conjecture \cite{BIR} in the multiplicative setting.
Let $S$ be a finite set of rational primes and $x\in \overline{\mathbb Q}$. Let $K=\mathbb Q(x)$. We denote by $M(K)=M^0(K)\cup M^\infty(K)$ the set of finite and infinite places of $K$. Let $$S_K=M^\infty(K)\cup \{\nu\in M^0(K): \nu|p\text{ for some }p\in S\}.$$ We call $x$ an $S$-unit if $|x|_\nu= 1$ for all $\nu\in M(K)\setminus S_K$.

Let $K$ be a number field and ${P\in K[X_1^{\pm 1},\dots,X_n^{\pm 1}]\setminus \{0\}}$. Ih's Conjecture predicts that the set of torsion points $\boldsymbol\zeta\in \mathbb G_m^n$ such that $P(\boldsymbol\zeta)$ is an $S$-unit is not Zariski dense in $\mathbb G_m^n$, unless the zero set of $P$ is a finite union of proper torsion cosets. Using refined versions of Theorems \ref{thm:atoral} and  \ref{thm:pSimple} we can show

\begin{theorem}\label{thm:Sunits}Let $S$ be a finite set of primes, $K\subseteq\mathbb{C}$ a number field and ${P\in K[X_1^{\pm1},\dots,X_n^{\pm1}]\setminus\{0\}}$.  Suppose that $\tau(P)$ is essentially atoral for all field embeddings $\tau: K\to \mathbb{C}$. Then the zero set of $P$ in $\mathbb{G}_m^n$ is  a finite union of torsion cosets and the coefficients of $P$ are $S$-integers, or there exists $B\geq 1$ such that if $\boldsymbol{\zeta}\in\mathbb{G}_m^n$ has finite order and $P(\boldsymbol{\zeta})$ is an $S$-unit, then $\boldsymbol{\zeta}^{\mathbf{a}}=1$ for some $\mathbf{a}\in\mathbb{Z}^n\setminus\{{0}\}$ such that $|\mathbf{a}|\leq B$.\end{theorem}

Since the set of torsion points $\boldsymbol\zeta$ such that $\delta(\boldsymbol\zeta)\le B$ is contained in a finite union of proper algebraic subgroups, it is not dense in $\mathbb G_m^n$ and hence Theorem \ref{thm:Sunits} proves Ih's Conjecture for polynomials all of whose conjugates are essentially atoral.

\paragraph{Overview of the proofs} In the univariate case we factor ${F=a_0(X-\alpha_1)\cdots(X-\alpha_n)}$. It turns out that only the $\alpha_k$ such that $|\alpha_k|_p=1$ contribute to the error. To understand $|\sigma(\zeta)-\alpha|_p$ for $\sigma \in \mathrm{Gal}(\mathbb Q(\zeta)/\mathbb Q)$ we pick $\sigma_0$ such that $\sigma_0(\zeta)$ is closest to $\alpha$. Then we write $$\sigma(\zeta)-\alpha= (\sigma(\zeta)-\sigma_0(\zeta))+(\sigma_0(\zeta)-\alpha).$$
If the order of $\sigma(\zeta)/\sigma_0(\zeta)$ has a prime divisor different from $p$ we have $|\sigma(\zeta)-\sigma_0(\zeta)|_p=1$. We use the ultrametric inequality to deduce $|\sigma(\zeta)-\alpha|_p=1$. In the other case, where the order of $\sigma(\zeta)/\sigma_0(\zeta)$ is a power of $p$, we can bound $|\sigma(\zeta)-\alpha|_p\ge |\sigma(\zeta)-\sigma_0(\zeta)|_p=|\sigma(\zeta)/\sigma_0(\zeta)-1|_p$ by the ultrametric inequality and since the absolute value of roots of unity equals one. Thus we have to understand $|\zeta-1|_p$ where the order of $\zeta$ is a power of $p$. To get our bound we use that $\prod_{i=1}^{n-1} (1-\zeta^i) = n$ for all roots of unity $\zeta$ of order $n$.

The proof in the multivariate case follows the lines of section 5 and 6 as well as the proof of Theorem 8.8 in \cite{Atoral}. We factor $\boldsymbol{\zeta=\xi\eta}$ and use Galois theory to reduce to the univariate case. Roughly speaking, if $\boldsymbol\xi=(\omega^{u_1},\dots,\omega^{u_n})$ for some root of unity $\omega$, the univariate polynomial is $Q_a(X)=F(\boldsymbol\eta^a X^{a\mathbf u^\top})$. One difference is that the complex $N$-th roots of unity are given by $e^{2\pi i k/N}, k=1,\dots, N$, but no similar representation does exist in $\mathbb C_p$. But the $N$-th roots of unity in $\mathbb C$ are isomorphic to those in $\mathbb C_p$ as groups, thus it is mainly a difference in notation. Instead of the Mahler measure, we have to deal with $\log|F|_p$. Although $m(\cdot)$ and $\log|\cdot |_p$ have some analogue properties, we need a new proof to relate the logarithm of the Gauss norms of the transformations with $\log|F|_p.$ Given a family $F_1,\dots, F_m\in\mathbb C_p[X_1,\dots, X_n]$ and $\mathbf u_1,\dots, \mathbf u_m\in\mathbb Z^r$ we show that there is a torsion point $\boldsymbol\zeta\in\mathbb G_m^r$ such that $F=\sum_{i=1}^m \boldsymbol\zeta^{\mathbf u_i}F_i$ satisfies $|F|_p=\max\{|F_i|_p: i=1,\dots,m\}$. To prove this we use the fact that for any $F\in \mathbb C_p[X]$ there are at most $\deg(F)$ roots of unity $\zeta$ of order prime to $p$ such that $|F(\zeta)|_p\ne|F|_p$.
\\

The idea of the application is that thanks to the product formula and the $S$-integrality of $P(\boldsymbol\zeta)$ we can find an equation of averages. Then we use the equidistribution results and conclude that either the limits must be the same or $\delta(\boldsymbol\zeta)$ is bounded. If the limits are the same we find that $m(Q)+\sum_{p\in\mathbb P}\log |Q|_p =0,$ where $Q=\prod_{\tau: K\to \mathbb C} \tau(P)$. Using additivity of the Mahler measure and the product formula we find that $m(Q')=0$ where $Q'\in\mathbb Z[X_1,\dots, X_n]$ is the primitive part of $Q$. We can use the characterisation of such polynomials which was proven by Boyd in \cite{Boyd} and see that $Q'$ is a product of extended cyclotomic polynomials. These are of the form $X_1^{a_1}\cdots X_n^{a_n}\phi_m(X_1^{b_1}\cdots X_n^{b_n})$ where $b_1,\dots, b_n$ are coprime integers and $\phi_m\in\mathbb Z[X]$ is the $m$-th cyclotomic polynomial. In particular the zero set of $Q'$ is a finite union of torsion cosets. 

\paragraph{Organisation of the article} After setting up the notations, we first prove Theorem \ref{thm:atoral} for univariate polynomials. Then we generalize it to the multidimensional setting in section \ref{sec:multVariate}.

For the application we first introduce the $p$-adic absolute values in number fields, and the $S$-units. Then we give equidistribution results in the archimedian as well as in the non-archimedian setting. Finally we study extended cyclotomic polynomials and prove Theorem \ref{thm:Sunits}.

\paragraph{Acknowledgements} I want to thank my advisor Philipp Habegger for helpful conversations and simplifications. This is part of my PhD thesis. I have
received funding from the Swiss National Science Foundation grant number 200020\_184623.

\section{Notation}\label{sec:notation} Let us first define some symbols. As usual $\mathbb{Z},\mathbb{Q}$ and $\mathbb{C}$ denote the rational integers, the rational numbers and the complex numbers. The natural numbers are $\mathbb{N}=\{1, 2, 3,\dots\}$ and we write $\mathbb{N}_0$ for $\mathbb{N}\cup\{0\}$. We denote by $\mathbb{P}=\{2,3,5,\dots\}$ the set of rational primes. Let $K$ be an algebraically closed field. For all $n\in\mathbb{N}$ we denote by $\mu_n=\{x\in K:x^n=1\}$ the $n$-th roots of unity in $K$. We denote Euler's totient function by $\varphi$. For $p\in \mathbb{P}$ we denote by $\mathbb{Q}_p$ the $p$-adic numbers, by $\overline{\mathbb{Q}_p}$ their algebraic closure and by $\mathbb{C}_p$ a completion of $\overline{\mathbb{Q}_p}$ which is both complete and algebraically closed. The $p$-adic absolute value on $\mathbb{Q}_p$ extends uniquely to $\mathbb{C}_p$ and this extension is also denoted by $|\cdot |_p.$ Finally we obtain a norm on $\mathbb{C}_p^n$ by setting $|\mathbf{x}|_p=\max\{|x_i|_p:i=1,\dots, n\}$ where $\mathbf{x}=(x_1,\dots, x_n)$. We write $\omega_n$ for a primitive $n$-th root of unity in $\mathbb{C}_p$. For a polynomial $F\in\mathbb{C}_p[X_1,\dots,X_n]$ and a multiindex $\boldsymbol{\alpha}=(\alpha_1,\dots,\alpha_n)\in \mathbb{N}_0^n$ we denote by $f_{\boldsymbol{\alpha}}$ the coefficient of $F$ at the monomial $X_1^{\alpha_1}\cdots X_n^{\alpha_n}$. For $n\in\mathbb{N}$ and a field $K$ we write $\mathbb{G}_m^n=\mathbb G_m^n(K)=(\overline{K}\setminus\{0\})^n$ which is a group with componentwise multiplication. We denote by $\mu_\infty^n$  its points of finite order. Let $R$ be a ring. For a vector of invertible elements $\mathbf{x}=(x_1,\dots, x_n)\in (R^*)^n$ and $A=(a_{i,j})\in \mathrm{Mat}_{n\times m}(\mathbb Z)$ we define $$\mathbf{x}^A=(x_1^{a_{1,1}}\cdots x_n^{a_{n,1}},\dots, x_1^{a_{1,m}}\cdots x_n^{a_{n,m}})\in (R^*)^m.$$
This notation is compatible with the multiplication of matrices in the following sense.
For all $n,r,s\in\mathbb{N}$, $\mathbf{x}\in (R^*)^n$, $A\in \mathrm{Mat}_{n\times r}(\mathbb Z)$ and $B\in\mathrm{Mat}_{r \times s}(\mathbb Z)$ we have $(\mathbf{x}^A)^B = \mathbf{x}^{AB}$. An element in $\mathbb Z^n$ is considered as a column vector. By $\left<\cdot\,,\cdot\right>$ we denote the scalar product in $\mathbb R^n$. If $S$ is a finite set, $\#S$ denotes the number of elements in it. We also use the Vinogradov notation $\ll_d$ and $\gg_d$, and the Landau notation $o$ and $O$.

\section{Univariate case}
Let $p$ be a prime. We give an estimate for $\left|\frac{1}{\#M}\sum_{\sigma\in M} \log|F(\sigma(\zeta))|_p-|F|_p\right|$ for a univariate polynomial $F\in\mathbb C_p[X]$; here $\zeta\in\mathbb G_m(\mathbb C_p)$ is a torsion point and $M$ is any subset of $\mathrm{Gal}(\mathbb Q(\zeta)/\mathbb Q)$. We extensively use the fact, that for $x,y\in\mathbb C_p$ the absolute value of $x+y$ is given by the maximum of their absolute values whenever $|x|_p\ne |y|_p$.

\begin{definition}Let $n\in\mathbb{N}$ and $F\in\mathbb{C}_p[X_1,\dots, X_n]$. Then we define $\mathrm{Supp}(F)=\{\boldsymbol{\alpha}\in\mathbb{N}_0^n: f_{\boldsymbol{\alpha}} \neq 0\}$ the support of $F$, the set of multiindices where the corresponding coefficient does not vanish. Further we denote by $|F|_p=\max \{|f_{\boldsymbol{\alpha}}|_p: \boldsymbol{\alpha}\in \mathbb{N}_0^n\}$ the $p$-adic maximum norm of $F$.\end{definition}

\begin{lemma}\label{lem:inf} Let $F\in \mathbb{C}_p[X_1,\dots, X_n]\setminus\{0\}$. Then $$\inf(F)=\inf\{|F(\boldsymbol\zeta)|_p:F(\boldsymbol\zeta)\neq 0, \boldsymbol\zeta\in \mu_\infty^n\} \text{ and }c(F)=-\log\left(|F|_p^{-1}\inf(F)\right)$$ are well defined and we have $c(F)\ge 0$. \end{lemma}
\begin{proof}The only polynomial which vanishes on every torsion point of $\mathbb{G}_m^n(\mathbb{C}_p)$ is the zero polynomial which we excluded, so the infimum is well defined. It is positive by Theorem \ref{thm:TV}, so we can take the logarithm of $|F|_p^{-1}\inf(F)$. Note that for every torsion point $\boldsymbol{\zeta}$ we have $\inf(F)\le |F(\boldsymbol\zeta)|_p\leq |F|_p$ by the ultrametric inequality. So $c(F)\geq 0$.\end{proof}

\begin{lemma}\label{lem:absDalphaF}Let $F\in \mathbb{C}_p[X]$ be a polynomial. Let $i\in\mathbb{N}_0$ and $x\in \mathbb{C}_p$ with $|x|_p\le 1$. Then $$\left|\frac{1}{i!}\frac{d^i F}{dX^i}(x)\right|_p \leq |F|_p.$$\end{lemma}\begin{proof}By linearity of the derivative and the ultrametric inequality we can assume that $F=X^n$. Let $i\in\mathbb{N}_0$. An inductive argument shows that we have $$\frac{d^i X^n}{d X^i} =\begin{cases}\frac{n!}{(n-i)!}X^{n-i}& \text{if } i \leq n,\\  0 &\text{else}.\end{cases}$$ Since binomial coefficients are integers we find $\left|\frac{1}{i!}\frac{d^iF}{dX^i}(x)\right|_p=\left|\binom{n}{i}x^{n-i}\right|_p\le 1=|F|_p$ in the first case. The inequality holds trivially in the second case.\end{proof}

\begin{lemma}\label{lem:absAlphaZeta}Let $F\in\mathbb{C}_p[X]\setminus\{0\}$. Suppose that $\alpha$ is a root of $F$ such that $|\alpha|_p=1$ and let $\zeta$ be a root of unity such that $F(\zeta)\neq 0$. Then we have $$|\zeta-\alpha|_p\geq \exp(-c(F)).$$\end{lemma}
\begin{proof}By the Taylor expansion we have $0=F(\alpha)=\sum_{i=0}^{\deg(F)}\frac{(\alpha-\zeta)^i}{i!} \frac{d^i F}{d X^i} (\zeta).$ Putting $F(\zeta)$ on the other side and dividing through $\alpha-\zeta$ yields $$\frac{-F(\zeta)}{\alpha-\zeta}=\sum_{i=1}^{\deg(F)}\frac{(\alpha-\zeta)^{i-1}}{i!} \frac{d^i F}{d X^i}(\zeta).$$ Taking absolute values and applying Lemma \ref{lem:absDalphaF} gives $$\left|\frac{F(\zeta)}{\alpha-\zeta}\right|_p\leq \max_{i=1,\dots,\deg(F)}|(\alpha-\zeta)|_p^{i-1}\left|\frac{1}{i!} \frac{d^i F}{d X^i} (\zeta)\right|_p\leq |F|_p.$$ Thus recalling the definition of $c(F)$ we find $\exp(-c(F))\leq |F|_p^{-1}|F(\zeta)|_p\leq |\zeta-\alpha|_p$.\end{proof}

\begin{lemma}\label{lem:errAbs1}Let $F=a_0\prod_{i=1}^{\deg(F)} (X-\alpha_i)\in \mathbb{C}_p[X]\setminus\{0\}$ be a polynomial and $\zeta\in\mathbb{C}_p$ a root of unity. Let $M$ be a non-empty subset of $\mathrm{Gal}(\mathbb{Q}(\zeta)/\mathbb{Q})$ and assume that $F(\sigma(\zeta))\neq 0$ for all $\sigma\in M$. Then we have
$$\frac{1}{\#M}\sum_{\sigma\in M} \log|F(\sigma(\zeta))|_p = \log|F|_p + \frac{1}{\#M}\sum_{\substack{i=1 \\ |\alpha_i|_p=1}}^{\deg(F)}\sum_{\sigma\in M} \log|\sigma(\zeta)-\alpha_i|_p.$$\end{lemma}
\begin{proof}We can rearrange the sum as follows \begin{align*}\frac{1}{\#M}\sum_{\sigma\in M} \log|F(\sigma(\zeta))|_p 
&= \frac{1}{\#M}\sum_{\sigma\in M}\left(\log |a_0|_p+ \sum_{i=1}^{\deg(F)}\log|\sigma(\zeta)-\alpha_i|_p \right)\\
&=\left(\log |a_0|_p+\frac{1}{\#M}\sum_{\sigma\in M}\sum_{\substack{i=1 \\ |\alpha_i|_p>1}}^{\deg(F)}\log|\sigma(\zeta)-\alpha_i|_p\right)\\ & \quad + \frac{1}{\#M}\sum_{\substack{i=1 \\ |\alpha_i|_p=1}}^{\deg(F)}\sum_{\sigma\in M} \log|\sigma(\zeta)-\alpha_i|_p+  \frac{1}{\#M}\sum_{\sigma\in M}\sum_{\substack{i=1 \\ |\alpha_i|_p<1}}^{\deg(F)}\log|\sigma(\zeta)-\alpha_i|_p.\end{align*}
We have $|\sigma(\zeta)|_p=1$ for all $\sigma\in M$. Therefore by a consequence of the ultrametric inequality we have $$\frac{1}{\#M}\sum_{\sigma\in M}\sum_{\substack{i=1 \\ |\alpha_i|_p>1}}^{\deg(F)}\log|\sigma(\zeta)-\alpha_i|_p= \sum_{\substack{i=1 \\ |\alpha_i|_p>1}}^{\deg(F)}\log|\alpha_i|_p \text{ and }\frac{1}{\#M}\sum_{\sigma\in M}\sum_{\substack{i=1 \\ |\alpha_i|_p<1}}^{\deg(F)}\log|\sigma(\zeta)-\alpha_i|_p=0.$$
By the Gauss lemma we have $$|F|_p= |a_0|_p\prod_{i=1}^{\deg(F)}|X-\alpha_i|_p= |a_0|_p\prod_{i=1}^{\deg(F)}\max\{1,|\alpha_i|_p\} = |a_0|_p\prod_{\substack{i=1 \\ |\alpha_i|_p>1}}^{\deg(F)}|\alpha_i|_p.$$ Combining the equations yields $$\frac{1}{\#M}\sum_{\sigma\in M} \log|F(\sigma(\zeta))|_p = \log|F|_p + \frac{1}{\#M}\sum_{\substack{i=1 \\ |\alpha_i|_p=1}}^{\deg(F)}\sum_{\sigma\in M} \log|\sigma(\zeta)-\alpha_i|_p.\qedhere$$\end{proof}

\begin{lemma}\label{lem:philIn1}Let $n\in \mathbb{N}$ and let $\zeta\in\mathbb{C}_p$ be a root of unity of order $n$. Then $$\prod_{i=1}^{n-1} (1-\zeta^i)=n.$$\end{lemma}
\begin{proof}Evaluate $(X-\zeta)\cdots (X-\zeta^{n-1}) = X^{n-1}+\dots+1$ in $X=1$.\end{proof}

\begin{lemma}\label{lem:notpPower}Let $l\neq p$ be a prime and $\zeta\in\mathbb{C}_p$ a root of unity whose order is divisible by $l$.
Then $|\zeta-1|_p =1$.\end{lemma}
\begin{proof}Let $n=lm$ be the order of $\zeta$. Since we have $|1-\zeta^k|_p\le 1$ for every $k\in\mathbb N$ any product of such factors is of absolute value at most one. Note that $\zeta^l$ is of order $m.$ Thus by Lemma \ref{lem:philIn1} we have $$\prod_{\substack{i=1\\ l\nmid i}}^{n-1} |1-\zeta^i|_p
=\left(\prod_{i=1}^{n-1}|1-\zeta^i|_p\right)\left(\prod_{i=1}^{m-1} \left|1-\zeta^{li}\right|_p\right)^{-1}=|n|_p|m|_p^{-1}=|l|_p=1.$$ Therefore all the inequalities must be equalities: we have $|1-\zeta^k|_p=1$ for all $1\le k < n$ which are not multiples of $l$. In particular we find $|1-\zeta|_p=1$.\end{proof}

\begin{lemma}\label{lem:alphaabs1}Let $F \in \mathbb{C}_p[X]\setminus\{0\}$ be a polynomial and $\alpha\in\mathbb{C}_p$ a root of it. Let $\zeta\in\mathbb{C}_p$ be a root of unity of order $p^k m,\; p\nmid m,\; k\in\mathbb N_0$ and $M$ a non empty subset of $\mathrm{Gal}(\mathbb{Q}(\zeta)/\mathbb{Q})$. Suppose that $|\alpha|_p =1$ and $F(\sigma(\zeta))\neq 0$ for all $\sigma \in M$. Then we have
$$0 \geq \sum_{\sigma\in M}\log|\sigma(\zeta)-\alpha|_p
 \geq -\left({c}(F)+k \log(p)\right).$$\end{lemma}
\begin{proof}Let $N=p^k m$. Observe that $|\sigma(\zeta) - \alpha|_p \leq \max\{|\sigma(\zeta)|_p,|\alpha|_p\}=1$ for all $\sigma\in M$. So each summand is nonpositive and the sum bounded by zero. For the lower bound we define $$\epsilon := \min\{|\sigma(\zeta)-\alpha|_p: \sigma\in M\}>0.$$ Fix $\sigma_0\in M$ such that $\epsilon = |\sigma_0(\zeta)-\alpha|_p.$  Using the observation above we see that $\epsilon\leq 1.$ We split up into the two cases $\epsilon =1$ and $\epsilon<1$. In the first case we have $|\sigma(\zeta)-\alpha|_p=1$ for all $\sigma\in M$ and so $$\sum_{\sigma\in G}\log|\sigma(\zeta)-\alpha|_p=0\geq -\left({c}(F)+k\log(p)\right).$$
In the second case we have $\epsilon <1$. The set $M\setminus \{\sigma_0\}$ is the disjoint union of 
\begin{align*}M_1 &= \{\sigma\in M: \mathrm{ord}(\sigma(\zeta)/\sigma_0(\zeta)) \text{ is not a power of } p\} \text{ and }\\
M_2 &= \{\sigma\in M: \sigma\neq \sigma_0 \text{ and ord}(\sigma(\zeta)/\sigma_0(\zeta)) |p^k\}.\end{align*}
Suppose $\sigma\in M_1$. By construction the order of $\sigma(\zeta)/\sigma_0(\zeta)$ is divisible by a prime $l\neq p$. Therefore by Lemma \ref{lem:notpPower} we have $|\sigma(\zeta)/\sigma_0(\zeta)-1|_p= 1$ and hence $|\sigma(\zeta)-\sigma_0(\zeta)|_p=1$. We use that $|\cdot|_p$ is non-archimedean and $|\sigma_0(\zeta)-\alpha|_p = \epsilon < 1 = |\sigma(\zeta)-\sigma_0(\zeta)|_p$ to deduce
$$|\sigma(\zeta)-\alpha|_p = |(\sigma(\zeta)-\sigma_0(\zeta)) + (\sigma_0(\zeta)-\alpha)|_p=\max\{|\sigma(\zeta)-\sigma_0(\zeta)|_p,|\sigma_0(\zeta)-\alpha|_p\} =1.$$ Thus the sum over $M_1$ is \begin{equation}\label{eq:G1}\sum_{\sigma\in M_1} \log |\sigma(\zeta)-\alpha|_p =0.\end{equation}
Now suppose $\sigma\in M_2$. If $|\sigma(\zeta)-\sigma_0(\zeta)|_p\neq \epsilon$, then
$$|\sigma(\zeta)-\alpha|_p=|(\sigma(\zeta)-\sigma_0(\zeta)) + (\sigma_0(\zeta)-\alpha)|_p=\max\{|\sigma(\zeta)-\sigma_0(\zeta)|_p,|\sigma_0(\zeta)-\alpha|_p\}$$ and in particular $|\sigma(\zeta)-\alpha|_p\geq |\sigma(\zeta)-\sigma_0(\zeta)|_p$. If $|\sigma(\zeta)-\sigma_0(\zeta)|_p = \epsilon$, then by definition of $\epsilon$ we find that $|\sigma(\zeta)-\alpha|_p\geq |\sigma(\zeta)-\sigma_0(\zeta)|_p$ holds also in this case. Therefore
$$\sum_{\sigma\in M_2} \log|\sigma(\zeta)-\alpha|_p\geq \sum_{\sigma\in M_2} \log|\sigma(\zeta)/\sigma_0(\zeta)-1|_p\geq \sum_{\xi\in \mu_{p^k}\setminus\{1\}} \log|\xi - 1|_p$$ by the definition of $M_2$. We may pick up extra terms in the final expression, but they are harmless as $|\xi-1|_p\leq 1$ for all $\xi\in \mu_{p^k}$. Observe that $\mu_{p^k}\setminus \{1\} = \{\rho^i: i=1,\dots, p^{k -1}\}$ for any root of unity $\rho\in\mathbb{C}_p$ of order $p^k$. So by Lemma \ref{lem:philIn1} we find \begin{equation}\label{eq:G2}\sum_{\sigma\in M_2} \log|\sigma(\zeta)-\alpha|_p\geq \log\left|\prod_{i=1}^{p^k -1} (1-\rho^i)\right|_p= \log|p^k|_p=k \log|p|_p = -k\log(p).\end{equation}
By Lemma \ref{lem:absAlphaZeta} we have $\epsilon=|\sigma_0(\zeta)-\alpha|_p\geq \exp(-c(F))$ and hence $\log(\epsilon)\geq -c(F)$.
We add the sums \eqref{eq:G1} and \eqref{eq:G2} over $M_1$ and $M_2$ respectively, to the remaining term $\log |\sigma_0 (\zeta) - \alpha|_p$ which equals $\log(\epsilon)$ and get $\sum_{\sigma\in M}\log|\sigma(\zeta)-\alpha|_p
 \geq -(c(F)+k \log(p))$. \end{proof}

\begin{lemma}\label{lem:univariate} Let $F \in \mathbb{C}_p[X]\setminus\{0\}$ and $\zeta\in\mathbb{C}_p$ a root of unity of order $p^k m, m\nmid p$. Let $M$ be a non empty subset of $\mathrm{Gal}(\mathbb{Q}(\zeta)/\mathbb{Q})$ such that $F(\sigma(\zeta))\neq 0$, for all $\sigma \in M.$ Then we have
$$\left|\frac{1}{\#M}\sum_{\sigma\in M} \log|F(\sigma(\zeta))|_p-\log|F|_p\right|\le \frac{\deg(F)}{\#M}\left(c(F)+k\log(p)\right).$$\end{lemma}
\begin{proof}Write $F=a_0\prod_{i=1}^{\deg(F)} (X-\alpha_i)$. By Lemma \ref{lem:errAbs1} we have that 
$$\frac{1}{\#M}\sum_{\sigma\in M} \log|F(\sigma(\zeta))|_p - \log|F|_p = \frac{1}{\#M}\sum_{\substack{i=1 \\ |\alpha_i|_p=1}}^{\deg(F)}\sum_{\sigma\in M} \log|\sigma(\zeta)-\alpha_i|_p.$$ Suppose that $|\alpha_i|=1.$ We apply Lemma
\ref{lem:alphaabs1} and find $\left|\sum_{\sigma\in M}\log|\sigma(\zeta)-\alpha_i|_p\right|\leq c(F) + k\log(p)$. Therefore we have $$\left|\frac{1}{\#M}\sum_{\substack{i=1 \\ |\alpha_i|_p=1}}^{\deg(F)}\sum_{\sigma\in M} \log|\sigma(\zeta)-\alpha_i|_p\right|\leq \frac{1}{\#M}\sum_{\substack{i=1 \\ |\alpha_i|_p=1}}^{\deg(F)}\left|\sum_{\sigma\in M} \log|\sigma(\zeta)-\alpha_i|_p\right|\leq \frac{\deg(F)}{\#M}\left(c(F)+k\log(p)\right).\qedhere$$\end{proof}

\section{Multivariate case}\label{sec:multVariate}
In this section we follow the same strategy as in $\cite{Atoral}$ to prove Theorem \ref{thm:padic} which will directly imply Theorem \ref{thm:pSimple}.

\subsection{$\Lambda_{\boldsymbol\zeta}$ and related definitions}
In this subsection we associate to a torsion point $\boldsymbol\zeta$ in $\mathbb G_m^n$ a discrete subgroup $\Lambda_{\boldsymbol\zeta}<\mathbb Z^n$ and compute its rank and determinant. The ground field is $\mathbb C_p$ for some prime $p$.

\begin{definition}Let $n\in\mathbb{N}$. A lattice $\Lambda$ in $\mathbb{R}^n$ is a finitely generated and discrete subgroup of $\mathbb{R}^n$. We denote by rk$(\Lambda)$ its rank and by $\det(\Lambda)$ the determinant given by $\det(\Lambda)=|\det(A)|^{1/2}$, where $A=\left(\left<\mathbf{v}_i,\mathbf{v}_j\right>\right)_{i,j}$ for a $\mathbb{Z}$-basis $\mathbf{v}_1,\dots, \mathbf{v}_m$ of $\Lambda$. \end{definition}

\begin{remark}\label{rk:det}Let $\mathbf v_1,\dots, \mathbf v_m$ be a basis of $\Lambda$. Let $V=(\mathbf v_1,\dots, \mathbf v_m)$. Then the determinant of $\Lambda$ is given by $|\det(V^\top V)|^{1/2}$. If $\mathbf w_1,\dots, \mathbf w_m$ is another basis of $\Lambda$ and $W=(\mathbf w_1,\dots, \mathbf w_m)$, then there exists $B\in\mathrm{GL}_m(\mathbb Z)$ such that $W=VB$. Thus we have $\det(W^\top W)=\det(B^\top V^\top VB)=\det(V^\top V)$ and hence the determinant is independent of the choice of the basis. Moreover, suppose that $\Lambda<\mathbb Z^n$ is a lattice of rank $n$ and $\mathbf b_1,\dots, \mathbf b_n$ is a basis of $\Lambda$. Let $B=(\mathbf b_1,\dots, \mathbf b_n)\in \mathrm{Mat}_n(\mathbb Z).$ Then we have $$\det(\Lambda)=|\det(B^\top B)|^{1/2}=|\det(B)|=[\mathbb Z^n:B\mathbb Z^n]=[\mathbb Z^n:\Lambda].$$\end{remark}

\begin{definition}Let $\Lambda\subseteq\mathbb{R}^n$ be a lattice with positive rank. We define $$\lambda_1(\Lambda) = \min\{|\mathbf u| : \mathbf u\in \Lambda\setminus\{{0}\}\}$$ where $|\cdot|$ denotes the maximum norm. It is convenient to define $\lambda_1(\{{0}\})=\infty$.\end{definition}

\begin{definition}For $\boldsymbol{\zeta}\in\mathbb{G}_m^n$ we define $\Lambda_{\boldsymbol{\zeta}}=\{\mathbf{u}\in\mathbb{Z}^n: \boldsymbol{\zeta}^{\mathbf{u}}=1\}.$\end{definition}

\begin{lemma}\label{lem:detRk}Let $\boldsymbol{\zeta}\in\mathbb{G}_m^n$ be of order $N$. Then $\Lambda_{\boldsymbol{\zeta}}$ is a lattice of rank $n$ and determinant $N$.\end{lemma}
\begin{proof}Since $\Lambda_{\boldsymbol \zeta}$ contains $N \mathbb Z^n$ , it has full rank. Consider the group homomorphism $\phi : \mathbb Z^n \mapsto \mathbb G_m$, given by $\phi(\mathbf a) = \boldsymbol{\zeta}^{\mathbf{a}}$. Then $\mathrm{ker}(\phi) = \Lambda_{\boldsymbol \zeta}$ and $\mathrm{Im}(\phi) = \mu_N$. By Remark \ref{rk:det} and the isomorphism theorem we find $\det(\Lambda_{\boldsymbol{\zeta}})=[\mathbb Z^n : \Lambda_{\boldsymbol \zeta}] = \#\mu_N = N$.
\end{proof}

In the introduction to Section 5 in \cite{Atoral} Dimitrov and Habegger take a lattice $\Lambda\ne \{0\}$ such that $\det(\Lambda)\ge 1$ and a parameter $\nu\in(0,1/2]$ and construct a lattice $\Lambda(\nu)$ which satisfies $\mathrm{rk}(\Lambda/\Lambda(\nu))\ge 1$.

\begin{definition}\label{def:lambda}Let $\boldsymbol{\zeta}\in\mathbb{G}_m^n$ be of order $N$ and $\nu\in(0,1/2]$ a parameter. By Lemma \ref{lem:detRk} $\Lambda_{\boldsymbol{\zeta}}$ has determinant $N\geq 1$, so $\Lambda_{\boldsymbol{\zeta}}(\nu)$ is defined and we can put $$\tilde{\Lambda}_{\boldsymbol{\zeta}}(\nu) = \{ u\in \mathbb{Z}^n: \text{ there is } k\in\mathbb{Z}\setminus\{0\} \text{ such that } ku\in\Lambda_{\boldsymbol{\zeta}}(\nu)\}.$$
For technical reasons we also define $$\tilde{\lambda}(\boldsymbol{\zeta};\nu) = \min\left\{\lambda_1(\tilde\Lambda_{\boldsymbol{\zeta}}(\nu)), N^{\nu^n/2}\right\}.$$\end{definition}

\subsection{Some Galois theory}
We state one of the main results in Galois theory which we will use a lot. We also show how a factorisation $\boldsymbol{\zeta=\xi\eta}$ can be used to rewrite an average as double sum.
\begin{theorem}\label{thm:lang}Let $K/k$ be a Galois extension and $F/k$ an arbitrary extension and assume that $K,F$ are subfields of some common field. Then $KF/F$ and $K/(K\cap F)$ are Galois extensions. Let $H=\mathrm{Gal}(KF/F)$ and $\sigma\in H$, then $\sigma|_K \in \mathrm{Gal}(K/(K\cap F))$ and the map $\sigma \mapsto \sigma|_K$ induces an isomorphism between $H$ and $\mathrm{Gal}(K/(K\cap F))$.\end{theorem}
\begin{proof}This is Theorem VI.1.12 in \cite{Lang}.\end{proof}

\begin{lemma}\label{lem:splitSum}Let $\boldsymbol{\zeta}\in\mathbb{G}_m^n(\mathbb{C}_p)$ be a torsion point of order $N$ and $l\in\{1,\dots,n-1\}$. Let $\star_1$ denote the multiplication on $\mathbb G_m^n$ and let $\star_2:\mathbb G_m^l\times \mathbb G_m^{n-l}\to \mathbb G_m^n$ given by $\boldsymbol\eta\star_2\boldsymbol\xi = (\boldsymbol\eta,\boldsymbol\xi).$ Let $\star\in\{\star_1, \star_2\}$ and $\boldsymbol\zeta\in \mathbb G_m^n$ a torsion point of order $N$. Let $\boldsymbol{\eta,\xi}$ be two torsion points whose orders divide $N$ such that $\boldsymbol\zeta=\boldsymbol\eta\star\boldsymbol\xi$. Let $G<\mathrm{Gal}(\mathbb{Q}(\boldsymbol{\zeta})/\mathbb{Q})$ and $L\subseteq\mathbb{Q}(\boldsymbol{\zeta})$ its fixed field. Let $H=\mathrm{Gal}(L(\boldsymbol{\xi})/L(\boldsymbol{\xi})\cap L(\boldsymbol{\eta}))$ and for all $\tau\in \mathrm{Gal}(L(\boldsymbol{\eta})/L)$ let $\tilde{\tau}\in\mathrm{Gal}(L(\boldsymbol{\xi})/L)$ be such that $\tilde{\tau}\restriction_{L(\boldsymbol{\xi})\cap L(\boldsymbol{\eta})} = \tau\restriction_{L(\boldsymbol{\xi})\cap L(\boldsymbol{\eta})}$. Finally let $f: \{\sigma(\boldsymbol\zeta):\sigma \in G\} \to \mathbb{R}$ be a function. Then we have $$\frac{1}{\#G}\sum_{\sigma\in G} f(\sigma(\boldsymbol\zeta)) = \frac{1}{[L(\boldsymbol{\eta}):L]}\sum_{\tau\in \mathrm{Gal}(L(\boldsymbol{\eta})/L)} \frac{1}{\#H} \sum_{\sigma\in\tilde{\tau}H} f(\tau(\boldsymbol\eta)\star\sigma(\boldsymbol\xi)).$$\end{lemma}
\begin{proof}Since the order of $\boldsymbol{\zeta}$ is $N$ we have $L(\boldsymbol{\zeta})=L(\omega_N)$ for a primitive $N$-th root of unity $\omega_N\in\mathbb C_p$. The orders of $\boldsymbol{\xi}$ and $\boldsymbol{\eta}$ divide $N$, so the coordinates are powers of $\omega_N$ and we find $L(\boldsymbol{\zeta})=L(\boldsymbol{\xi},\boldsymbol{\eta})$. Because $L\subseteq\mathbb{Q}(\boldsymbol{\zeta})$ we have $\mathbb{Q}(\boldsymbol{\zeta})=L(\boldsymbol{\zeta})$. By definition of $L$ we have ${G=\mathrm{Gal}(\mathbb{Q}(\boldsymbol{\zeta})/L)=\mathrm{Gal}(L(\boldsymbol{\zeta})/L)=\mathrm{Gal}(L(\boldsymbol{\xi,\eta})/L)}$. Applying Theorem \ref{thm:lang} to $L(\boldsymbol{\xi})$ and $L(\boldsymbol{\eta})$ shows that $\psi:\mathrm{Gal}(L(\boldsymbol{\xi,\eta})/L(\boldsymbol{\eta}))\to \mathrm{Gal}(L(\boldsymbol{\xi})/L(\boldsymbol{\xi})\cap L(\boldsymbol{\eta}))=H$, $\sigma \mapsto \sigma \restriction_{L(\boldsymbol{\xi})}$ is an isomorphism. Consider the map \begin{align*}\phi: G&\to M=\{(\sigma,\tau)\in \mathrm{Gal}(L(\boldsymbol{\xi})/L)\times \mathrm{Gal}(L(\boldsymbol{\eta})/L): \sigma\restriction_{L(\boldsymbol{\xi})\cap L(\boldsymbol{\eta})} = \tau\restriction_{L(\boldsymbol{\xi})\cap L(\boldsymbol{\eta})}\}\\
\sigma &\mapsto (\sigma\restriction_{L(\boldsymbol{\xi})},\sigma\restriction_{L(\boldsymbol{\eta})}).\end{align*} We now want to show that $\phi$ is a bijection. So let $(\sigma, \tau)\in M$. Let $\hat{\tau}\in\mathrm{Gal}(L(\boldsymbol{\xi,\eta})/L)$ be an extension of $\tau$ and $\tilde{\tau}$ the restriction of $\hat{\tau}$ to $L(\boldsymbol{\xi})$. Since $\sigma\restriction_{L(\boldsymbol{\xi})\cap L(\boldsymbol{\eta})} = \tau\restriction_{L(\boldsymbol{\xi})\cap L(\boldsymbol{\eta})}=\tilde{\tau}\restriction_{L(\boldsymbol{\xi})\cap L(\boldsymbol{\eta})}$, the automorphism $\sigma\tilde{\tau}^{-1}\in\mathrm{Gal}(L(\boldsymbol{\xi})/L)$ restricts to the identity on $L(\boldsymbol{\xi})\cap L(\boldsymbol{\eta})$ and therefore $\sigma\tilde{\tau}^{-1}\in H$. Since $\psi$ is surjective, there exists $\hat{\sigma}\in \mathrm{Gal}(L(\boldsymbol{\xi,\eta})/L(\boldsymbol{\eta}))$ such that $\hat{\sigma}\restriction_{L(\boldsymbol{\xi})}=\sigma\tilde{\tau}^{-1}$. We have $\hat{\sigma}\hat{\tau}\in G$ and $$\hat{\sigma}\hat{\tau}\restriction_{L(\boldsymbol{\xi})}=\hat{\sigma}\restriction_{L(\boldsymbol{\xi})}\hat{\tau}\restriction_{L(\boldsymbol{\xi})}= \sigma\tilde{\tau}^{-1}\tilde{\tau}=\sigma \text{ and } \hat{\sigma}\hat{\tau}\restriction_{L(\boldsymbol{\eta})}=\hat{\sigma}\restriction_{L(\boldsymbol{\eta})}\hat{\tau}\restriction_{L(\boldsymbol{\eta})}= \text{id}_{L(\boldsymbol{\eta})}\tau=\tau,$$ so $\phi$ is surjective. To see that $\phi$ is also injective, let $\sigma,\tau\in G$ such that $\phi(\sigma)=\phi(\tau)$. Since the coordinates of $\boldsymbol{\xi}$ are in $L(\boldsymbol{\xi})$ and those of $\boldsymbol{\eta}$ in $ L(\boldsymbol{\eta})$, we find $\sigma(\boldsymbol{\xi})=\tau(\boldsymbol{\xi})$ and $\sigma(\boldsymbol{\eta})=\tau(\boldsymbol{\eta})$. Because any automorphism in $G$ is determined by its values at the coordinates of $\boldsymbol{\xi}$ and $\boldsymbol{\eta}$ we conclude $\sigma=\tau$. Note also that if $\phi(\sigma)=(\tau,\tau')$ then $\sigma(\boldsymbol{\zeta})=\sigma(\boldsymbol{\xi})\star\sigma(\boldsymbol{\eta})=\tau(\boldsymbol{\xi})\star\tau'(\boldsymbol{\eta})$. So we can write $$\sum_{\sigma\in G} f(\sigma(\boldsymbol{\zeta})) =\sum_{(\sigma,\tau)\in M} f(\sigma(\boldsymbol{\xi})\star\tau(\boldsymbol{\eta}))=\sum_{\tau\in \mathrm{Gal}(L(\boldsymbol\eta)/L)} \sum_{\substack{\sigma\in \mathrm{Gal}(L(\boldsymbol{\xi})/L)\\\sigma\restriction_{L(\boldsymbol{\xi})\cap L(\boldsymbol{\eta})} = \tau\restriction_{L(\boldsymbol{\xi})\cap L(\boldsymbol{\eta})}}} f(\tau(\boldsymbol{\eta})\star\sigma(\boldsymbol{\xi})).$$ 

We claim that $$\{\sigma\in \mathrm{Gal}(L(\boldsymbol{\xi})/L): \sigma\restriction_{L(\boldsymbol{\xi})\cap L(\boldsymbol{\eta})} = \tau\restriction_{L(\boldsymbol{\xi})\cap L(\boldsymbol{\eta})} \} = \tilde{\tau}H.$$ We have $\sigma\restriction_{L(\boldsymbol{\xi})\cap L(\boldsymbol{\eta})} = \tau\restriction_{L(\boldsymbol{\xi})\cap L(\boldsymbol{\eta})}=\tilde{\tau}\restriction_{L(\boldsymbol{\xi})\cap L(\boldsymbol{\eta})}$ if and only if $\sigma\tilde{\tau}^{-1}\in\mathrm{Gal}(L(\boldsymbol{\xi})/L)$ restricts to the identity on $L(\boldsymbol{\xi})\cap L(\boldsymbol{\eta})$ if and only if $\sigma\tilde{\tau}^{-1}\in H$ which is equivalent to $\sigma\in\tilde{\tau}H$. To see that $\#G = [L(\boldsymbol{\eta}):L] \#H$, we consider the group homomorphism $G\to \mathrm{Gal}(L(\boldsymbol{\eta})/L)$ given by restriction to $L(\boldsymbol{\eta})$. The kernel of this surjective map is $\mathrm{Gal}(L(\boldsymbol{\xi,\eta})/L(\boldsymbol{\eta}))$ which is isomorphic to $H$ under $\psi$.\end{proof}

\subsection{Preparatory results}
The main goal of this subsection is to prove bounds in any dimension, under some condition which ensures that the Gauss norm of the univariate polynomial we construct equals $|F|_p$. We also prove some statements about the Gauss norm.
\begin{definition}For $N\in \mathbb{N}$ let $\Gamma_N = \left(\mathbb{Z}/N\mathbb{Z}\right)^*$ denote the invertible elements modulo $N$. If $M|N$ we have a canonical and surjective group homomorphism $\Gamma_N\to \Gamma_M$ given by reduction. The conductor $f_G$ of a subgroup $G<\Gamma_N$ is defined as $$f_G=\min\{f: f|N\text{ and } \ker (\Gamma_N\to \Gamma_f)<G\}.$$ This is well defined since the kernel of $\Gamma_N\to\Gamma_N$ is trivial and hence always contained in $G$.\end{definition}

\begin{lemma}\label{lem:conductor}Let $N$ be a natural number and $\zeta\in\mathbb{C}_p$ a root of unity of order $N$. Then $\Gamma_N$ is naturally isomorphic to $\mathrm{Gal}(\mathbb{Q}(\zeta)/\mathbb{Q})$. Let $G<\Gamma_N$. Through the isomorphism it makes sense to speak about the fixed field $L_G\subseteq \mathbb{Q}(\zeta)$ of $G$. Let $f\geq 1$ be an integer and denote by $\omega_f$ a root of unity of order $f$ in $\mathbb C_p$. Then $L_G\subseteq \mathbb{Q}(\omega_f)$ if and only if $f_G | f$. In particular the conductor $f_G$ depends only on the fixed field $L_G$.\end{lemma}
\begin{proof}This is proven for $K=\mathbb C$ in the introduction to Section 3 in \cite{Atoral}. The same argument works for $K=\mathbb C_p$.\end{proof}

\begin{definition}Let $n\in\mathbb{N}$ and $V=(v_{i,j})_{i,j=1}^n\in \mathrm{Mat}_n(\mathbb{Z})$. Then we denote by $|\cdot |$ the maximum norm $|V|=\max\{|v_{i,j}|: i,j=1,\dots, n\}.$\end{definition}

\begin{lemma}\label{lem:absVinv}Let $n\in\mathbb{N}$ and $V\in \text{GL}_n(\mathbb{Z})$. Then $|V^{-1}|\leq (n-1)! |V|^{n-1}$.\end{lemma}
\begin{proof}If $n=1$ we have $V=V^{-1}=\pm 1$, and the norm satisfies $|V^{-1}|=1=0!1^0=0!|V|^0$. If $n\geq 2$ and since the determinant is $\pm1$, the inverse $V^{-1}$ is given by the adjoint up to the sign. Its entries are up to the sign given by the determinant of a submatrix of $V$ formed by deleting a row and a column. By the formula of Leibniz such a determinant is given by the sum over $(n-1)!$ terms which each is a product of $n-1$ entries of $V$ up to sign. Each entry is bounded by $|V|$, so each summand by $|V|^{n-1}$ and hence $|V^{-1}|\leq (n-1)!|V|^{n-1}$.\end{proof}

\begin{lemma}\label{lem:decAtoral}Let $\nu\in (0,1/4]$ and let $\boldsymbol{\zeta}\in \mathbb{G}_m^n$ be of order $N$. Let $r=\mathrm{rk } (\Lambda_{\boldsymbol{\zeta}}/\Lambda_{\boldsymbol{\zeta}}(\nu))\in \{1,\dots, n\}$. Then there exists $V\in \text{GL}_n(\mathbb{Z})$ and a decomposition $\boldsymbol{\zeta=\eta \xi}$ with $\boldsymbol{\eta, \xi}$ in $\mathbb{G}_m^n$ of order $E$ and $M$, respectively such that the following holds
\begin{enumerate}\item $E\mid N, M\mid N$ and $E\leq N^{2\nu^{1+r}}$. In particular, $\mathbb{Q}(\boldsymbol{\zeta})=\mathbb{Q}(\boldsymbol{\xi,\eta})$ holds.
\item We have $|V|\ll_d N^{2v^{1+r}}$ with $\boldsymbol{\xi}^V=(1,\dots, 1,\boldsymbol{\xi'})$ and $\boldsymbol{\xi'}\in \mathbb{G}_m^r(\mathbb{C}_p)$.
\item If $G<\mathrm{Gal}(\mathbb{Q}(\boldsymbol{\xi})/\mathbb{Q})$ there exist $\mathbf{a}\in \mathbb{Z}^r$ and $\omega_M\in\mathbb{C}_p$ of order $M$ such that $\boldsymbol{\xi'} = \omega_M^{\mathbf{a}^\top}$, $$|\mathbf{a}|<M \text{ and } \frac{|\mathbf{a}|}{M} \ll_d \frac{[\mathrm{Gal}(\mathbb{Q}(\boldsymbol{\xi})/\mathbb{Q}):G]f_G^{1/2}}{N^{\nu^r/(6n)}}.$$
\item We have $\delta(\boldsymbol{\xi})\geq n^{-1/2}\tilde{\lambda}(\boldsymbol{\zeta};\nu)$.\end{enumerate}\end{lemma}
\begin{proof}Recall that $\mathbf {e(x)}=(e^{2\pi i x_1},\dots, e^{2\pi i x_m})\in \mathbb G_m^m$ for all $\mathbf x=(x_1+\mathbb Z,\dots, x_m+\mathbb Z)\in (\mathbb R/\mathbb Z)^m$ and also for all $\mathbf x=(x_1,\dots, x_m)\in\mathbb R^m$. Let $\omega_N\in\mathbb{C}_p$ be a root of unity of order $N$ and $\psi$ be an isomorphism between $\mathbb{Q}(\mathbf e(1/N))$ and $\mathbb{Q}(\omega_N)$. Then since the order of $\boldsymbol{\zeta}$ equals $N$, all its coordinates are in the image of $\psi$ and therefore it makes sense to look at the preimage $\tilde{\boldsymbol{\zeta}}=\psi^{-1}(\boldsymbol{\zeta})\in \mathbb{G}_m^n$. which also has order $N$. Since $\Lambda_{\boldsymbol{\zeta}}=\Lambda_{\tilde{\boldsymbol{\zeta}}}$ we have $\text{rk } \Lambda_{\tilde{\boldsymbol{\zeta}}}/\Lambda_{\tilde{\boldsymbol{\zeta}}}(\nu)= r$. By Proposition 5.1 in \cite{Atoral} there exists  $V\in \text{GL}_n(\mathbb{Z})$ and a decomposition $\boldsymbol{\tilde\zeta=\tilde\eta \tilde\xi}$ with $\boldsymbol{\tilde\eta, \tilde\xi}$ in $\mathbb{G}_m^n$ of order $E$ and $M$, respectively such that the following holds
\begin{enumerate}\item $E\mid N, M\mid N$ and $E\leq N^{2\nu^{1+r}}$. In particular, $\mathbb{Q}(\tilde{\boldsymbol{\zeta}})=\mathbb{Q}(\boldsymbol{\tilde\xi,\tilde\eta})$ holds.
\item We have $|V|\ll_d N^{2v^{1+r}}$ with $\boldsymbol{\tilde\xi}^V=(1,\dots, 1,\boldsymbol{\tilde\xi'})$ and $\boldsymbol{\tilde\xi'}\in \mathbb{G}_m^r$.
\item If $G<\Gamma_M$ there exist $\mathbf{a}\in \mathbb{Z}^r$ and $\sigma\in G$ such that $\boldsymbol{\tilde\xi'} = \mathbf e(\mathbf{a}\sigma/M)$, $$|\mathbf{a}|<M \text{ and } \frac{|\mathbf{a}|}{M} \ll_d \frac{[(\mathbb{Z}/M\mathbb{Z})^*:G]f_G^{1/2}}{N^{\nu^r/(6n)}}.$$
\item We have $\delta(\boldsymbol{\tilde\xi})\geq n^{-1/2}\tilde{\lambda}(\boldsymbol{\tilde\zeta};\nu)$.\end{enumerate}
Let $\boldsymbol{\xi}=\psi(\boldsymbol{\tilde\xi})$ and $\boldsymbol{\eta}=\psi(\boldsymbol{\tilde\eta})$. Then $\boldsymbol{\zeta=\eta \xi}$ and the orders of $\boldsymbol{\eta, \xi}$ are $E$ and $M$, respectively since the order does not change under an isomorphism and neither do $\delta$ and $\tilde{\lambda}$. Now 1,2 and 4 follow directly after applying $\psi$. If we apply $\psi$ to 3. we find $\boldsymbol{\xi'}=\psi(\mathbf e(\mathbf{a}\sigma/M))=\psi(\mathbf e(\sigma/M))^{\mathbf{a}}$ where $\psi(\mathbf e(\sigma/M))\in\mathbb C_p$ is a root of unity of order $M$.\end{proof}

\begin{definition} For $\mathbf{u}\in\mathbb{Z}^n$ we define $$\rho(\mathbf{u})=\inf\{|\mathbf{v}|:\mathbf{v}\in\mathbb{Z}^n\setminus\{{0}\} \text{ and } \left<\mathbf{u},\mathbf{v}\right>=0\}.$$ We use the convention $\inf \emptyset = \infty$.\end{definition}

\begin{lemma}\label{lem:pend62}Let $p\in\mathbb{P}, \boldsymbol{\zeta}\in \mathbb{G}_m^n(\mathbb{C}_p)$ be of finite order $N$. Let $G<\mathrm{Gal}\left(\mathbb{Q}(\boldsymbol{\zeta})/\mathbb{Q}\right)$. Let $F$ in $\mathbb{C}_p[X_1,\dots,X_n]$ be a polynomial such that $F(\sigma(\boldsymbol{\zeta}))\neq 0$ for all $\sigma\in G$. Then the following holds:
If $n=1$ we have $$\frac{1}{\#G}\sum_{\sigma\in G} \log |F(\sigma(\boldsymbol{\zeta}))|_p = \log |F|_p + O_\epsilon\left(\frac{\deg(F)[\mathrm{Gal}(\mathbb{Q}(\boldsymbol{\zeta})/\mathbb{Q}): G]^2f_G^{1/2}\left(1+c(F)\right)}{N^{1-\epsilon}}\right).$$
If $n\geq 2,\; \tilde{\lambda}(\boldsymbol{\zeta};\nu)>n^{1/2}\deg(F),\; 0<\nu\leq 1/(126n^2)$ and $r=n-\mathrm{rk} (\Lambda_{\boldsymbol{\zeta}}(\nu))\geq 1$ we have $$\frac{1}{\#G}\sum_{\sigma\in G} \log |F(\sigma(\boldsymbol{\zeta}))|_p = \log |F|_p + O_{n,\nu}\left(\frac{\deg(F)[\mathrm{Gal}(\mathbb{Q}(\boldsymbol{\zeta})/\mathbb{Q}):G]^2f_G^{1/2}\left(1+c(F)\right)}{N^{\nu^r/(12n)}}\right).$$\end{lemma}

\begin{proof}The case $n=1$ follows directly from Lemma \ref{lem:univariate} since $\log(N)\ll_\epsilon N^\epsilon$, $\varphi(N)\gg_\epsilon N^{1-\epsilon}$ and $1/\#G = [\mathrm{Gal}(\mathbb{Q}(\boldsymbol{\zeta})/\mathbb{Q}):G]/\#\mathrm{Gal}(\mathbb{Q}(\boldsymbol{\zeta})/\mathbb{Q}) = [\mathrm{Gal}(\mathbb{Q}(\boldsymbol{\zeta})/\mathbb{Q}):G]/\varphi(N)$.\\
If $n\geq 2$ we follow the strategy in \cite{Atoral} to reduce it to the univariate case. This idea was inspired by a work of Duke \cite{Duke}. Let us denote by $L$ the fixed field of $G$ in $\mathbb{Q}(\boldsymbol{\zeta})$. We apply Lemma \ref{lem:decAtoral} to factor $\boldsymbol{\zeta}$. Hence there are $V\in \text{GL}_n(\mathbb{Z})$ and $\boldsymbol{\eta,\xi}\in \mathbb{G}_m^n$ of order $E$ and $M$ respectively such that $\boldsymbol{\zeta=\eta\xi}$, $E\leq N^{2\nu^{1+r}}$ and therefore $M\geq N/E\geq N^{1-2\nu^{1+r}}$. Furthermore $|V|\ll_n N^{2\nu^{1+r}}$ and $\boldsymbol{\xi}^V=(1,\dots, 1,\boldsymbol{\xi'})$ for some $\boldsymbol{\xi'}\in \mathbb{G}_m^r$. Let $H=\mathrm{Gal}(\mathbb{Q}(\boldsymbol{\xi})/\mathbb{Q}(\boldsymbol{\xi})\cap L(\boldsymbol{\eta}))$, it is the image under the restriction map $\mathrm{Gal}(L(\boldsymbol{\xi})/L) \to \mathrm{Gal}(\mathbb{Q}(\boldsymbol{\xi})/\mathbb{Q}(\boldsymbol{\xi})\cap L)$ of $\mathrm{Gal}(L(\boldsymbol{\xi})/L(\boldsymbol{\xi})\cap L(\boldsymbol{\eta}))$. Since this restriction map is an isomorphism by Theorem \ref{thm:lang}, $H$ is isomorphic to $\mathrm{Gal}(L(\boldsymbol{\xi})/L(\boldsymbol{\xi})\cap L(\boldsymbol{\eta}))$ and we can identify them. Since $F(\sigma(\boldsymbol{\zeta}))\neq 0$ for all $\sigma\in G$, the function $\log|F(\cdot)|_p$ is well defined on the orbit of $\boldsymbol{\zeta}$ under $G$. Applying Lemma \ref{lem:splitSum} to $G$, the factorisation $\boldsymbol{\zeta=\eta\xi}$ and this map we get $$\frac{1}{\#G}\sum_{\sigma\in G} \log|F(\sigma(\boldsymbol{\zeta}))|_p = \frac{1}{[L(\boldsymbol{\eta}):L]}\sum_{\tau\in \mathrm{Gal}(L(\boldsymbol{\eta})/L)} \frac{1}{\#H} \sum_{\sigma\in\tilde{\tau}H} \log|F(\tau(\boldsymbol{\eta})\sigma(\boldsymbol{\xi}))|_p,$$ where $\tilde{\tau}\in\mathrm{Gal}(L(\boldsymbol{\xi})/L)$ is such that $\tilde{\tau}\restriction_{L(\boldsymbol{\xi})\cap L(\boldsymbol{\eta})} = \tau\restriction_{L(\boldsymbol{\xi})\cap L(\boldsymbol{\eta})}$. For $\tau\in \mathrm{Gal}(L(\boldsymbol{\eta})/L)$ let $S_\tau$ denote the inner sum $$S_\tau= \frac{1}{\#H}\sum_{\sigma\in \tilde{\tau}H}\log |F(\tau(\boldsymbol{\eta})\sigma(\boldsymbol{\xi}))|_p= \frac{1}{\#H}\sum_{\sigma\in H}\log |F(\tau(\boldsymbol{\eta})\tilde{\tau}\sigma(\boldsymbol{\xi}))|_p.$$ So we can write \begin{equation}\label{eq:decStau}\frac{1}{\#G}\sum_{\sigma \in G} \log |F(\sigma(\boldsymbol{\zeta}))|_p=\frac{1}{[L(\boldsymbol{\eta}):L]}\sum_{\tau \in \mathrm{Gal}(L(\boldsymbol{\eta})/L)} S_\tau.
\end{equation}
Applying the third statement of Lemma \ref{lem:decAtoral} to $H$ gives us $\mathbf{a}\in\mathbb{Z}^r$ such that $$|\mathbf{a}|<M \text{ and } \frac{|\mathbf{a}|}{M}\ll_n \frac{[\mathrm{Gal}(\mathbb{Q}(\boldsymbol{\xi})/\mathbb{Q}): H]f_H^{1/2}}{N^{\nu^r/(6n)}},$$ and $\omega_M\in\mathbb{C}_p$ such that $\boldsymbol{\xi'} = \omega_M^{\mathbf{a}^\top}$. Let $\mathbf{0}\in \mathbb Z^{n-r}$ and define $\mathbf{u}^\top=\mathbf{\left(0^\top,a^\top\right)} V^{-1}\in \mathbb{Z}^n$. Then we have $$\boldsymbol{\xi}=(\boldsymbol{\xi}^V)^{V^{-1}}=(\mathbf{1,\dots,1},\boldsymbol{\xi'})^{V^{-1}}= (\omega_M^{\mathbf{0}^\top},\omega_M^{\mathbf{a}^\top})^{V^{-1}} = \omega_M^{\mathbf{u}^\top}.$$ For all $\tau\in \text{Gal}(L(\boldsymbol{\eta})/L)$ let $Q_\tau=F\left(\tau(\boldsymbol{\eta})X^{\mathbf{u}^\top}\right)X^{v_\tau}\in \mathbb C_p[X]$, where we choose $v_\tau\in\mathbb{Z}$ in such a way that $Q_\tau$ is a polynomial and $Q_\tau(0)\neq 0$. This is possible because  we have $|Q_\tau(\tilde{\tau}(\omega_M))|_p=\left|F\left(\tau(\boldsymbol{\eta})\tilde{\tau}(\omega_M)^{\mathbf{u}^\top}\right)\right|_p = |F(\tau(\boldsymbol{\eta})\tilde{\tau}(\boldsymbol{\xi}))|_p\neq 0$ since $\tau(\boldsymbol{\eta})\tilde{\tau}(\boldsymbol{\xi})=\sigma(\boldsymbol{\zeta})$ for some $\sigma\in G$. Since $\mathrm{Gal}(L(\boldsymbol{\xi})/L)$ is an abelian group we can write $$S_\tau=\frac{1}{\#H}\sum_{\sigma\in H} \log |Q_\tau(\sigma(\tilde{\tau}(\omega_M)))|_p.$$

We want to apply the univariate result to $Q_\tau$ and the root of unity $\tilde{\tau}(\omega_M)$ of order $M$. Therefore we bound a couple of quantities which will arise:
The fixed field of $H$ is contained in $L(\boldsymbol{\eta})$, therefore $[\mathrm{Gal}(\mathbb{Q}(\boldsymbol{\xi})/\mathbb{Q}):H]\leq [L(\boldsymbol{\eta}):\mathbb{Q}]\leq [L:\mathbb{Q}]E$ holds. Since $L$ is the fixed field of $G$, we find $[L:\mathbb{Q}]=[\mathrm{Gal}(\mathbb{Q}(\boldsymbol{\zeta})/\mathbb{Q}):G]$, so $$[\mathrm{Gal}(\mathbb{Q}(\boldsymbol{\xi})/\mathbb{Q}):H]\leq [\mathrm{Gal}(\mathbb{Q}(\boldsymbol{\zeta})/\mathbb{Q}):G]N^{2\nu^{1+r}}.$$ Next we want to bound the conductor of $H$. Therefore we want to find a field generated by a root of unity which contains $\mathbb{Q}(\boldsymbol{\xi})\cap L(\boldsymbol{\eta})$ its fixed field. By Lemma \ref{lem:conductor}, $L$ is contained in $\mathbb{Q}(\omega_{f_G})$. So $\mathbb{Q}(\boldsymbol{\xi})\cap L(\boldsymbol{\eta})\subseteq \mathbb{Q}(\boldsymbol{\xi})\cap \mathbb{Q}(\omega_{f_G},\boldsymbol{\eta})$. The last intersection is generated by a root of unity of order $\gcd(M,\text{lcm}(f_G,E)).$ So we have $$f_H\leq \text{lcm}(f_G,E)\leq f_G E \leq f_G N^{2\nu^{1+r}}.$$
Now let us bound the degree of $Q_\tau$. A monomial $X_1^{\alpha_1}\cdots X_n^{\alpha_n}$ in $F$ becomes up to a factor $X^{\left<\boldsymbol{\alpha},\mathbf{u}\right>+v_\tau}$ in $Q_\tau$. We have $|\left<\boldsymbol{\alpha},\mathbf{u}\right>|\leq \sum_{k=1}^n \alpha_k |u_k|\leq \deg(F)|\mathbf{u}|$, where $\mathbf{u}=(u_1,\dots, u_n)^\top$. Since $v_\tau$ is chosen in such a way that $Q_\tau$ is a polynomial and $Q_\tau (0)\neq 0$, there exists $\boldsymbol{\alpha}$ in $\mathrm{Supp}(F)$ such that $v_\tau=-\left<\boldsymbol{\alpha},\mathbf{u}\right>$. Therefore $v_\tau$ is also bounded by $\deg(F)|\mathbf{u}|$ and hence we find ${\deg(Q_\tau)\leq 2\deg(F)|\mathbf{u}|}$. Using Lemma \ref{lem:absVinv} and $M\leq N$ we have $$|\mathbf{u}|\ll_n|\mathbf{a}||V^{-1}|\ll_n \frac{|\mathbf{a}|}{M} N |V|^{n-1}\ll_n {[\mathrm{Gal}(\mathbb{Q}(\boldsymbol{\xi})/\mathbb{Q}):H]f_H^{1/2}}{N^{1+2(n-1)\nu^{r+1}-\nu^r/(6n)}}.$$ Putting all this bounds together we find \begin{align*}\deg(Q_\tau)&\ll \deg(F)|\mathbf{u}|\ll_n \deg(F)[\mathrm{Gal}(\mathbb{Q}(\boldsymbol{\xi})/\mathbb{Q}):H]f_H^{1/2}N^{1+2(n-1)\nu^{r+1}-\nu^r/(6n)}\\
&\ll_n \deg(F)[\mathrm{Gal}(\mathbb{Q}(\boldsymbol{\zeta})/\mathbb{Q}):G]f_G^{1/2}N^{1+2\nu^{r+1}+\nu^{r+1}+2(n-1)\nu^{r+1}-\nu^r/(6n)}.\end{align*}
The exponent of $N$ equals $1+(2n+1)\nu^{r+1}-\nu^r/(6n).$
We also have $$1/\#H=[\mathrm{Gal}(\mathbb{Q}(\boldsymbol{\xi})/\mathbb{Q}):H]/\varphi(M)\leq [\mathrm{Gal}(\mathbb{Q}(\boldsymbol{\zeta})/\mathbb{Q}):G]N^{2\nu^{r+1}}/\varphi(M).$$ We are now prepared to apply Lemma \ref{lem:univariate} to $Q_\tau$ and $\tilde{\tau}(\omega_M)$: we have \begin{align*}S_\tau&=\log|Q_\tau|_p + O\left(\frac{\deg(Q_\tau)}{\#H} \left(c(Q_{\tau})+\log(M)\right)\right)\\
&=\log|Q_\tau|_p + O_n\left(\frac{\deg(F)[\mathrm{Gal}(\mathbb{Q}(\boldsymbol{\zeta})/\mathbb{Q}):G]^2f_G^{1/2}N^{1+(2n+3)\nu^{r+1}-\nu^r/(6n)}}{\varphi(M)}  \left(c(Q_{\tau})+\log(N)\right)\right).\end{align*}
Let us bound $\varphi(M)\gg_{n,\nu} M^{1-\nu^r/(72n)}\geq N^{1-2\nu^{r+1}-\nu^r/(72n)}$ and $\log(N)\ll_{n,\nu} N^{\nu^r/(72n)}$. Then the exponent of $N$ is $1+(2n+3)\nu^{r+1}-\nu^r/(6n)+\nu^r/(72n)-(1-2\nu^{r+1}-\nu^r/(72n))$. Using $\nu\leq 1/(126n^2)$ we can bound the exponent from above by $\nu^r/n(7/126+1/36-1/6)=-\nu^r/(12n)$. Thus we get
$$S_\tau = \log|Q_\tau|_p + O_{n,\nu}\left(\frac{\deg(F)[\mathrm{Gal}(\mathbb{Q}(\boldsymbol{\zeta})/\mathbb{Q}):G]^2f_G^{1/2}\left(1+c(Q_{\tau})\right)}{N^{\nu^r/(12n)}}\right).$$

Lastly we want to understand $|Q_\tau|_p$ and ${c}(Q_\tau)$. We claim that $|Q_\tau|_p=|F|_p$ and $c(Q_\tau)\leq c(F)$. We prove the first claim by verifying that the coefficients of $Q_\tau$ are up to roots of unity the coefficients of $F$. Let $X_1^{\alpha_1}\dots X_n^{\alpha_n}$ and $X_1^{\beta_1}\dots X_n^{\beta_n}$ be monomials appearing in $F$ and suppose that $\boldsymbol{\alpha}\neq\boldsymbol{\beta}$. Then the corresponding monomials in $Q_\tau$ are $X^{\left<\boldsymbol{\alpha}, \mathbf{u}\right>+v_\tau}$ and $X^{\left<\boldsymbol{\beta}, \mathbf{u}\right>+v_\tau}$. If this monomials are the same we have $\left<\boldsymbol{\alpha-\beta}, \mathbf{u}\right>=0$ and since $|\boldsymbol{\alpha-\beta}|\leq \deg(F)$ we have $\rho(\mathbf{u})\leq\deg(F)$. This shows that if $\rho(\mathbf{u})>\deg(F)$, the coefficients of $Q_\tau$ are exactly those of $F$ up to a root of unity. Let $\mathbf{w}\in \mathbb{Z}^n\setminus\{0\}$ such that $\left<\mathbf{u,w}\right>=0$ and $|\mathbf{w}|=\rho(\mathbf{u})$. Using that $\boldsymbol{\xi}=\omega_M^{\mathbf{u} \top}$ we get $\boldsymbol{\xi}^{\mathbf{w}}=1$ and hence $\delta(\boldsymbol{\xi})\leq |\mathbf{w}|=\rho(\mathbf{u})$. With the bound from Lemma \ref{lem:decAtoral} we have $\rho(\mathbf{u})\geq n^{-1/2}\tilde{\lambda}(\boldsymbol{\zeta};\nu)$ which together with the assumption implies $\rho(\mathbf{u})>\deg(F)$. So we have $|Q_\tau|_p=|F|_p$.\\
Now we show the second claim. Let $\mu\in\mathbb{C}_p$ be a root of unity such that $Q_\tau(\mu)\neq 0$. Then $0\neq|Q_\tau(\mu)|_p=\left|F\left(\tau(\boldsymbol{\eta})\mu^{\mathbf{u}^\top}\right)\mu^{v_\tau}\right|_p=\left|F\left(\tau(\boldsymbol{\eta})\mu^{\mathbf u^\top}\right)\right|_p$ holds and $\tau(\boldsymbol{\eta})\mu^{\mathbf{u}^\top}\in\mu_\infty^n$. Therefore we have $\inf(F)\leq \inf(Q_\tau)$ using the notation of Lemma \ref{lem:inf}. This implies $c(Q_\tau)\leq c(F)$ since $|Q_\tau|_p=|F|_p$. 

Thus we have
$$S_\tau = \log|F|_p + O_{n,\nu}\left(\frac{\deg(F)[\mathrm{Gal}(\mathbb{Q}(\boldsymbol{\zeta})/\mathbb{Q}):G]^2f_G^{1/2}\left(1+c(F)\right)}{N^{\nu^r/(12n)}}\right).$$
Now we can take the average over $\tau$ and use \eqref{eq:decStau} to find $$\frac{1}{\#G}\sum_{\sigma\in G} \log|F(\sigma(\zeta))|_p= \log|F|_p + O_{n,\nu}\left(\frac{\deg(F)[\mathrm{Gal}(\mathbb{Q}(\boldsymbol{\zeta})/\mathbb{Q}):G]^2f_G^{1/2}\left(1+c(F)\right)}{N^{\nu^r/(12n)}}\right).\qedhere$$
\end{proof}

\begin{lemma}\label{lem:atoral87}Suppose $\boldsymbol{\zeta}\in \mathbb{G}_m^n$ has order $N$ and let $\delta\geq 1,\; \epsilon\in (0,1/2],\; \nu_1,\dots, \nu_{n-1}\in (0,1/2]$ with $\nu_1+\dots+\nu_{n-1}\leq 1/2$. Then there exist $l\in \{0,\dots, n-1\}$ and $V\in\mathrm{GL}_n(\mathbb{Z})$ such that the following hold
\begin{enumerate}
\item We have $|V|\ll_n \delta^{2\epsilon^{n-l}}$ and $V$ is the unit matrix if $l=0$.
\item \label{it:at872}We have $\boldsymbol{\zeta}^V=(\boldsymbol{\eta,\xi})$ where $\boldsymbol{\eta}\in \mathbb{G}_m^l, \boldsymbol{\xi}\in\mathbb{G}_m^{n-l}, \mathrm{ord}(\boldsymbol{\eta})\leq N^{\nu_1+\dots+\nu_l}$ and either $l=n-1$ and $\xi\in\mathbb{G}_m$ has order at least $N^{1/2}$ or $l\leq n-2$ and $\lambda_1(\tilde{\Lambda}_{\boldsymbol{\xi}}(\nu_{l+1}))> \delta^{\epsilon^{n-l-1}}$.\end{enumerate}\end{lemma}
\begin{proof}Since all quantities involved behave well under an isomorphism between $\mu_N$ and the $N$-th roots of $\mathbb{C}_p$, the lemma follows by Lemma 8.7 in \cite{Atoral}. \end{proof}

\begin{lemma}\label{lem:largePrimesOk}Let $F\in\mathbb{C}_p[X]\setminus\{0\}$. Then there are at most $\deg(F)$ roots of unity $\zeta$ of order prime to $p$ such that $|F(\zeta)|_p\neq |F|_p$.\end{lemma}
\begin{proof}Without loss of generality we can assume that $|F|_p=1$. Let $O=\{x\in\mathbb C_p: |x|_p\le 1\}$ be the ring of integers and $M=\{x\in O: |x|_p<1\}$ its maximal ideal. Note that $F\in O[X]$ and denote by $G$ its reduction modulo $M$. If $x\in O$ satisfies $|F(x)|_p<1$ its reduction is a root of $G$. Since $|F|_p=1$ we have $G\ne 0$ and hence it has at most $\deg(F)$ roots. The order of the quotient of two distinct roots of unity of order prime to $p$ is divisible by some prime $l\ne p$. Hence Lemma \ref{lem:notpPower} implies that the reduction modulo $M$ is injective on the roots of unity of order prime to $p$ and the lemma follows.
\end{proof}

\begin{lemma}\label{lem:specialComb}Let $F_1,\dots, F_m\in\mathbb{C}_p[X_1,\dots, X_n]$ be polynomials, $\mathbf{u}_1,\dots, \mathbf{u}_m\in \mathbb{Z}^r$ pairwise distinct, $\boldsymbol{\eta}\in\mu_\infty^n$ and $G<\mathrm{Gal}(\mathbb{Q}(\boldsymbol{\eta})/\mathbb{Q})$ such that for all $\tau\in G$ there is $1\leq i\leq m$ such that $F_{i}(\tau(\boldsymbol{\eta}))\neq 0$. Then there exists $\boldsymbol{\zeta}\in\mu_\infty^r$ such that $F=\boldsymbol{\zeta}^{\mathbf{u}_1}F_1+\dots +\boldsymbol{\zeta}^{\mathbf{u}_m}F_m$ satisfies $F(\tau(\boldsymbol{\eta}))\neq 0$ for all $\tau\in G$ and $|F|_p=\max\{|F_i|_p:1\leq i \leq m\}$.\end{lemma}

\begin{proof}If $m=1$ we can take any $\boldsymbol\zeta\in \mu_\infty^r$. So let us assume that $m\ge 2$ and define $R=\max_{1\leq i, j\leq m} |\mathbf{u}_i-\mathbf{u}_j|$. Let $\mathbf{v}\in\mathbb{Z}^r$ be such that $\rho(\mathbf{v})>R$. We will prove, that there exists $\zeta\in \mu_\infty$ such that $\zeta^{\mathbf{v}}$ has the desired properties. To do so, we show that both the non-vanishing properties and the equality of norms are satisfied for all but finitely many $\zeta$ of order prime to $p$.

For all $i=1,\dots, m$ let $n_i=\left<\mathbf{u}_i,\mathbf{v}\right>$. Since $\rho(\mathbf{v})>R\geq |\mathbf{u}_i-\mathbf{u}_j|$ we have $\left<\mathbf{u}_i-\mathbf{u}_j,\mathbf{v}\right>\neq 0$ for $i\neq j$ which implies that the $n_i$ are pairwise distinct. Let us define
$$F_\zeta=(\zeta^{\mathbf{v}})^{\mathbf{u}_1}F_1+\dots +(\zeta^{\mathbf{v}})^{\mathbf{u}_m}F_m=\zeta^{n_1}F_1+\dots +\zeta^{n_m}F_m.$$
For $\tau\in G$ we define $F_{\tau}=F_1(\tau(\boldsymbol{\eta}))X^{n_1}+\dots+F_m(\tau(\boldsymbol{\eta}))X^{n_m}$. Then clearly we have $F_\zeta(\tau(\boldsymbol{\eta}))=F_\tau(\zeta)$. Since the $n_i$ are pairwise distinct, the coefficients of $F_{\tau}$ are given by $F_i(\tau(\boldsymbol{\eta}))$. By hypothesis at least one of them is non zero and therefore $F_\tau\not=0$. Therefore the number of roots of $F_\tau$ is finite for every $\tau\in G$ and $F_\zeta(\tau(\boldsymbol{\eta}))\neq 0$ for all $\tau\in G$ for all but finitely many $\zeta$.\\
Let $D$ be an upper bound of the degrees of the $F_i$ and write $F_i=\sum_{\|\boldsymbol{\alpha}\|\leq D} f_{i,\boldsymbol{\alpha}} \mathbf{X}_n^{\boldsymbol{\alpha}}$ for all $i$, where $\|\boldsymbol{\alpha}\|=\alpha_1+\dots+\alpha_n$. Let us define $$P_{\boldsymbol{\alpha}}=X^{n_1}f_{1,\boldsymbol{\alpha}}+\dots + X^{n_m}f_{m,\boldsymbol{\alpha}}.$$ Then the coefficients of $F_\zeta$ are given by $P_{\boldsymbol{\alpha}}(\zeta).$ Now let $\zeta\in\mu_\infty$ have order prime to $p$. By Lemma \ref{lem:largePrimesOk} we have $|P_{\boldsymbol{\alpha}}(\zeta)|_p=|P_{\boldsymbol{\alpha}}|_p$ for all but finitely many such $\zeta$. For those we have $$|F_\zeta|_p=\max\{|P_{\boldsymbol{\alpha}}|_p:  \|\boldsymbol{\alpha}\|\leq D\}=\max\{|f_{i,\boldsymbol{\alpha}}|_p:  1\leq i \leq m,\; \|\boldsymbol{\alpha}\|\leq D\}=\max\{|F_i|_p: 1\leq i\leq m\}.$$
There are infinitely many numbers prime to $p$, and hence we find a $\zeta$ satisfying all the conditions.
\end{proof}

\begin{definition}\label{def:FVz}For $n\in\mathbb N$ we denote by $\mathbf X_n$ the vector of $n$ variables $(X_1,\dots, X_n)\in \mathbb C_p[X_1,\dots,X_n]$.
For $F\in \mathbb{C}_p[X_1,\dots, X_n]$ and $V\in\mathrm{GL}_d(\mathbb{Z})$ we set $F_V=F(\mathbf{X}_n^{V^{-1}})\mathbf{X}_n^{\mathbf{v}}$, where we choose $\mathbf{v}\in\mathbb{Z}^d$ such that $F_V$ is a polynomial coprime to $X_1\cdots X_n$. Let $l\in\{0,\dots, n-1\}$. For $\mathbf{z}=(z_1,\dots,z_l)\in\mathbb{C}_p^l$ we set $$F_{V,\mathbf{z}}=F_V(z_1,\dots, z_l, X_1,\dots, X_{n-l})\in \mathbb{C}_p[X_1,\dots, X_{n-l}]$$ In case $l=0$ this means $F_{V,\mathbf{z}}=F_V$.
\end{definition}

\subsection{Main theorem}
In this subsection we prove the main theorem. The proof follows the lines of the proof of Theorem 8.8 in \cite{Atoral}.
\begin{theorem}\label{thm:padic}Let $F\in\mathbb{C}_p[X_1,\dots,X_n]\setminus\{0\}$. There are constants $C=C(n)\geq 1$, $\kappa=\kappa(n)>0$ with the following property: Let $\boldsymbol{\zeta}\in \mathbb{G}_m^n$ have order $N$ and suppose $G<\mathrm{Gal}(\mathbb{Q}(\boldsymbol{\zeta})/\mathbb{Q})$ such that $F(\sigma(\boldsymbol{\zeta}))\neq 0$ for all $\sigma\in G.$ If $\delta(\boldsymbol{\zeta})\geq C\deg(F)^C$ then $$\frac{1}{\#G}\sum_{\sigma\in G} \log |F(\sigma(\boldsymbol{\zeta}))|_p = \log |F|_p + O_{n}\left(\frac{\deg(F)[\mathrm{Gal}(\mathbb Q(\boldsymbol\zeta)/\mathbb Q):G]^2f_G^{1/2}\left(1+{c}(F)\right)}{\delta(\boldsymbol{\zeta})^\kappa}\right).$$\end{theorem}
\begin{proof}
The case $n=1$ follows directly from Lemma \ref{lem:pend62} since $\delta(\zeta)=N$. So we may assume $n\geq 2$. We work with the parameters $\nu_1,\dots,\nu_{n-1}\in (0,1/(126n^2)]$ and $\epsilon\in (0, 1/2]$. They are assumed to be small in terms of $n$ but independent of $p, F$ and $\boldsymbol{\zeta}$. We may assume that $\epsilon\leq \nu_l^n/2$ for all $l$. We determine them during the argument. Let us apply Lemma \ref{lem:atoral87} to $\boldsymbol{\zeta}, \delta=\delta(\boldsymbol{\zeta}), \epsilon$ and the $\nu_l$. Say $l,V,\boldsymbol{\eta}$ and $\boldsymbol{\xi}$ are given by this lemma. So among other things we know that $\boldsymbol{\zeta}^V=(\boldsymbol{\eta,\xi})$, $|V|\ll_n \delta(\boldsymbol{\zeta})^{2\epsilon^{n-l}}$ and $\mathrm{ord}(\boldsymbol{\eta})\leq N^{\nu_1+\dots+\nu_l}$.\\
If $l=0$, then $V$ is the unit matrix, $\boldsymbol{\xi}=\boldsymbol{\zeta}$ and $\lambda_1(\tilde{\Lambda}_{\boldsymbol{\zeta}}(\nu_1))>\delta(\boldsymbol{\zeta})^{\epsilon^{n-1}}$ since we are in the second case of Lemma \ref{lem:atoral87} 2) as $n-l=n\geq 2$. We always have $\delta(\boldsymbol{\zeta})\leq N$ and therefore we get $\tilde{\lambda}(\boldsymbol{\zeta},\nu_1)\geq \delta(\boldsymbol{\zeta})^{\min\{\epsilon^{n-1}, \nu_1^n/2\}}$. If $C>n^{1/2}\epsilon^{1-n}$ we get $\tilde{\lambda}(\boldsymbol\zeta,\nu_1)\geq n^{1/2}\deg(F)$ which was defined in Definition \ref{def:lambda}. So we can apply Lemma \ref{lem:pend62} to $\boldsymbol{\zeta},\nu_1$ and $F$ and we are left with the case $l\geq 1$.\\
Let $L$ be the fixed field of $G$ in $\mathbb{Q}(\boldsymbol{\zeta})$. Let $H=\mathrm{Gal}(\mathbb{Q}(\boldsymbol{\xi})/\mathbb{Q}(\boldsymbol{\xi})\cap L(\boldsymbol{\eta}))$ which is isomorphic to $\mathrm{Gal}(L(\boldsymbol{\xi})/L(\boldsymbol{\xi})\cap L(\boldsymbol{\eta}))$ by Galois theory. For details see the proof of Lemma \ref{lem:pend62}. For every $\tau \in\mathrm{Gal}(L(\boldsymbol\eta)/L)$ let $\tilde \tau\in\mathrm{Gal}(L(\boldsymbol\xi)/L)$ be such that the restrictions of $\tau$ and $\tilde\tau$ on $L(\boldsymbol\xi)\cap L(\boldsymbol\eta)$ coincide. Recall that for $\mathbf z\in\mathbb G_m^l$ we defined $F_{V,\mathbf z}\in\mathbb C_p[X_1,\dots, X_{n-l}]$ in Definition \ref{def:FVz}. Then by Lemma \ref{lem:splitSum} we have 
\begin{equation}\label{eq:innerOuter}\frac{1}{\#G}\sum_{\sigma \in G} \log |F(\sigma(\boldsymbol{\zeta}))|_p=\frac{1}{[L(\boldsymbol\eta):L]}\sum_{\tau \in \text{Gal}(L(\boldsymbol{\eta})/L)} \frac{1}{\#H}\sum_{\sigma\in H} \log |F_{V,\tau(\boldsymbol{\eta})}(\sigma(\tilde{\tau}(\boldsymbol{\xi})))|_p.\end{equation}
We will apply Lemma \ref{lem:pend62} to every average in the sum, thus to the subgroup $H<\mathrm{Gal}(\mathbb Q(\boldsymbol\xi)/\mathbb Q)$, the polynomial $F_{V,\tau(\boldsymbol\eta)}$ and the torsion point $\tilde\tau(\boldsymbol\xi)$. So let $\tau \in\mathrm{Gal}(L(\boldsymbol\eta)/L)$ and let $M, E$ denote the orders of $\boldsymbol\xi$ and $\boldsymbol\eta$ respectively. Then the order of $\tilde\tau(\boldsymbol\xi)$ equals $M$ and the order of $\tau(\boldsymbol\eta)$ is $E$.
Let us bound the index and the conductor of $H$. The fixed field of $H$ is contained in $L(\boldsymbol{\eta})$, therefore $$[\mathrm{Gal}(\mathbb Q(\boldsymbol\xi)/\mathbb Q):H]\leq [L(\boldsymbol{\eta}):\mathbb{Q}]\leq [L:\mathbb{Q}]E= [\mathrm{Gal}(\mathbb Q(\boldsymbol\zeta)/\mathbb Q):G]E\leq [\mathrm{Gal}(\mathbb Q(\boldsymbol\zeta)/\mathbb Q):G]N^{\nu_1+\dots+\nu_l}$$ holds. By Lemma \ref{lem:conductor}, $L$ is contained in $\mathbb{Q}(\omega_{f_G})$. So $\mathbb{Q}(\boldsymbol{\xi})\cap L(\boldsymbol{\eta})\subseteq \mathbb{Q}(\boldsymbol{\xi})\cap \mathbb{Q}(\omega_{f_G},\boldsymbol{\eta})$. The last intersection is generated by a root of unity of order $\gcd(M,\mathrm{lcm}(f_G,E)).$ So by Lemma \ref{lem:conductor} we have $$f_H\leq \text{lcm}(f_G,E)\leq f_G E \leq f_G N^{\nu_1+\dots+\nu_l}.$$
Next we give a lower bound for $M$ and show therefore that $N\leq ME$. This follows if we can show $\boldsymbol{\zeta}^{ME}=\mathbf{1}$. We have $$\boldsymbol{\zeta}^{ME}=(\boldsymbol{\eta,\xi})^{V^{-1}ME}=((\boldsymbol{\eta,\xi})^{ME})^{V^{-1}}=\mathbf{1}^{V^{-1}}=\mathbf{1}.$$ Hence we can bound $M\geq N/E\geq N^{1-\nu_1-\dots -\nu_l}\geq N^{1/2}$.

We bound $\deg(F_{V,\tau(\boldsymbol\eta)})$ and $c(F_{V,\tau(\boldsymbol{\eta})})$. Let $\mathbf v=(v_1,\dots, v_n)^\top\in\mathbb Z^n$ such that $F_V=F(\mathbf{X}_n^{V^{-1}})\mathbf{X}_n^{\mathbf{v}}$. For any $\mathbf{a}\in\mathrm{Supp}(F)$ there is a term $\mathbf{X}_n^{V^{-1}\mathbf{a+v}}$ in $F_V$. Let $1\le i\le n$. Since $F_V$ is coprime to $X_1\cdots X_n$ there exists $\mathbf a\in \mathrm{Supp}(F)$ such that the $i$-th coordinate of $V^{-1}\mathbf{a+v}$ vanishes. Thus we have $|v_i|\le |V^{-1}\mathbf a|$ and hence $|\mathbf v|\le \max\{|V^{-1}\mathbf a|: \mathbf a\in\mathrm{Supp}(F)\}$ which implies that $\deg(F_V) \ll_n \max\{|V^{-1}\mathbf{a}|: \mathbf{a}\in\mathrm{Supp}(F)\}$. We find $$\deg(F_{V,\tau(\boldsymbol{\eta})}) \leq \deg(F_V) \ll_n \max\{|V^{-1}\mathbf{a}|: \mathbf{a}\in\mathrm{Supp}(F)\}\ll_n |V|^{n-1}\deg(F)\ll_n \delta(\boldsymbol\zeta)^{2\epsilon^{n-l}n}\deg(F)$$ using Lemma \ref{lem:absVinv}.
To see that $$c(F_{V,\tau(\boldsymbol{\eta})})\leq c(F)$$ we show $|F_{V,\tau(\boldsymbol{\eta})}|_p\leq |F|_p$ and $\inf(F)\leq \inf(F_{V,\tau(\boldsymbol{\eta})}).$ 
The coefficients of $F_{V,\tau(\boldsymbol{\eta})}$ are polynomials in the coordinates of $\tau(\boldsymbol{\eta})$ whose coefficients are coefficients of $F$. The ultrametric triangle inequality then implies that $|F_{V,\tau(\boldsymbol{\eta})}|_p\leq |F|_p$. Now let $\boldsymbol{\mu}\in\mu_\infty^{n-l}$ such that $F_{V,\tau(\boldsymbol{\eta})}(\boldsymbol{\mu})\neq 0$. Then we have $0\neq|F_{V,\tau(\boldsymbol{\eta})}(\boldsymbol{\mu})|_p=|F_V(\tau(\boldsymbol\eta),\boldsymbol\mu)|_p= \left|F\left((\tau(\boldsymbol\eta),\boldsymbol\mu)^{V^{-1}}\right)(\tau(\boldsymbol\eta),\boldsymbol\mu)^{\mathbf{v}}\right|_p$ which is at least $\inf(F)$ since $(\tau(\boldsymbol\eta),\boldsymbol\mu)^{V^{-1}}$ is a torsion point in $\mathbb{G}_m^n$. This implies the inequality between the infimas and finishes the bounds of the various terms occuring in the error.\\

We split again in the two cases $l\leq n-2$ and $l=n-1$. If $l\leq n-2$ we claim that $\tilde{\lambda}(\tilde\tau(\boldsymbol{\xi});\nu_{l+1})>n^{1/2}\deg(F_{V,\tau(\boldsymbol{\eta})})$. Since $\Lambda_{\boldsymbol\xi}=\Lambda_{\tilde\tau(\boldsymbol\xi)}$ we have $\tilde{\lambda}(\tilde\tau(\boldsymbol{\xi});\nu_{l+1})=\tilde{\lambda}(\boldsymbol{\xi};\nu_{l+1})$. We know from Lemma \ref{lem:atoral87} that $\lambda_1(\tilde{\Lambda}_{\boldsymbol{\xi}}(\nu_{l+1}))> \delta(\boldsymbol{\zeta})^{\epsilon^{n-l-1}}$. Since $\delta(\boldsymbol{\zeta})\leq N$ we find $\tilde{\lambda}(\boldsymbol{\xi};\nu_{l+1})\geq\delta(\boldsymbol{\zeta})^{\min\{\epsilon^{n-l-1}, \nu_{l+1}^n/2\}}=\delta(\boldsymbol{\zeta})^{\epsilon^{n-l-1}}$. On the other hand there is a function $c(n)$ such that $\deg(F_{V,\tau(\boldsymbol{\eta})}) \leq c(n)\delta(\boldsymbol{\zeta})^{2\epsilon^{n-l}n}\deg(F)$. So it suffices to show $n^{1/2}c(n)\delta(\boldsymbol{\zeta})^{2\epsilon^{n-l}n}\deg(F)<\delta(\boldsymbol{\zeta})^{\epsilon^{n-l-1}}$. If $\epsilon$ is small enough in terms of $n$ and $C$ big enough, this is the case. Let us now apply Lemma \ref{lem:pend62} to $\tilde\tau(\boldsymbol{\xi}), \nu_{l+1},F_{V,\tau(\boldsymbol{\eta})}$ and $H$. We find that
\begin{equation}\label{eq:innerSum}\begin{aligned}
\frac{1}{\#H}&\sum_{\sigma\in H} \log |F_{V,\tau(\boldsymbol{\eta})}(\sigma(\tilde\tau(\boldsymbol{\xi})))|_p = \log |F_{V,\tau(\boldsymbol{\eta})}|_p \\
&\quad+ O_{n,\nu_{l+1}}\left(\frac{\deg(F_{V,\tau(\boldsymbol{\eta})})[\mathrm{Gal}(\mathbb Q(\boldsymbol\xi)/\mathbb Q):H]^2f_H^{1/2}\left(1+c(F_{V,\tau(\boldsymbol{\eta})})\right)}{M^{\nu_{l+1}^n/(12n)}}\right).
\end{aligned}\end{equation}
Let us now look at the case when $l=n-1$. Since $1/(12n) < 1$ we can apply Lemma \ref{lem:pend62} to $\boldsymbol\xi, \epsilon=1-1/(12n)$ and $F_{V,\boldsymbol \eta}$ and get a similar error where we have $1$ instead of $\nu_{l+1}^n$. So let us define $\nu_n=1$. Then the equation \eqref{eq:innerSum} above holds for all $l$.
Using the various bounds incuding $M\ge N^{1/2}$ we can bound the error term by $$O_{n,\nu_{l+1}}\left(\frac{\delta(\boldsymbol{\zeta})^{2\epsilon^{n-l}n}\deg(F) [\mathrm{Gal}(\mathbb Q(\boldsymbol\zeta)/\mathbb Q):G]^2f_G^{1/2} N^{(2+1/2)(\nu_1+\dots+\nu_l)}\left(1+c(F)\right)}{N^{\nu_{l+1}^n/(24n)}}\right).$$ We use the trivial bound $ \delta(\boldsymbol{\zeta})\leq N$ and we assume that the parameters satisfy $5/2(\nu_1+\dots+\nu_l) + 2\epsilon^{n-l}n \leq \frac{\nu_{l+1}^n}{48n}$. Now fix such parameters $\nu_j$ in terms of $n$. We may later shrink $\epsilon$. Then we get rid of the dependency of the constant in $\nu_{l+1}$ and for all $\kappa\leq \frac{\nu_{l+1}^n}{48n}$ we have
$$\frac{1}{\#H}\sum_{\sigma\in H} \log |F_{V,\tau(\boldsymbol{\eta})}(\sigma(\tilde\tau(\boldsymbol{\xi})))|_p = \log |F_{V,\tau(\boldsymbol{\eta})}|_p + O_{n}\left(\frac{\deg(F) [\mathrm{Gal}(\mathbb Q(\boldsymbol\zeta)/\mathbb Q):G]^2f_G^{1/2}\left(1+c(F)\right)}{\delta(\boldsymbol{\zeta})^\kappa}\right).$$ 
Note that the estimates do not depend on $\tau$. Thus pluging in the bounds in \eqref{eq:innerOuter} we find
\begin{equation}\label{eq:bdFmu}\begin{aligned}\frac{1}{\#G}\sum_{\sigma \in G} \log |F(\sigma(\boldsymbol{\zeta}))|_p
&=\left(\frac{1}{[L(\boldsymbol{\eta}):L]}\sum_{\tau \in \text{Gal}(L(\boldsymbol{\eta})/L)} \log |F_{V,\tau(\boldsymbol{\eta})}|_p\right)\\
&\quad + O_{n}\left(\frac{\deg(F) [\mathrm{Gal}(\mathbb Q(\boldsymbol\zeta)/\mathbb Q):G]^2f_G^{1/2}\left(1+c(F)\right)}{\delta(\boldsymbol\zeta)^\kappa}\right).\end{aligned}\end{equation}
Our final task now is to see, that the average over the $\log|F_{V,\tau(\boldsymbol{\eta})}|_p$ is close to $\log |F|_p$. Recall that $|F_{V,\tau(\boldsymbol{\eta})}|_p\leq |F|_p$. So we have to find a good lower bound for the average. Let us abbreviate $S=\mathrm{Supp}(F)$, so we can write $F_V=\sum_{\mathbf{a}\in S} f_{\mathbf{a}} \mathbf{X}_n^{V^{-1}\mathbf{a+v}}$. For all $\mathbf{a}\in S$ let $\mathbf{v_a}\in\mathbb{Z}^l$ and $\mathbf{u_a}\in\mathbb{Z}^{n-l}$ be such that $V^{-1}\mathbf{a+v}=\begin{pmatrix}\mathbf{v_a }\\ \mathbf{u_a}\end{pmatrix}$. Recall that $\mathbf{X}_l=(X_1,\dots, X_l)$ and let $\mathbf{Z}=(X_{l+1},\dots,X_n)$.  Let further $U=\{\mathbf{u_a : a}\in S\}$ and $S_{\mathbf{u}} = \{ \mathbf{a}\in S : \mathbf{u_a=u}\}$ and $F_{\mathbf{u}} = \sum_{\mathbf{a}\in S_\mathbf{u}} f_{\mathbf{a}}\mathbf{X}_l^{\mathbf{v_a}}$ for all $\mathbf{u}\in U$. Then we have
$$F_V= \sum_{\mathbf{u}\in U}\sum_{\mathbf{a}\in S_{\mathbf{u}}}f_{\mathbf{a}} \mathbf{X}_n^{V^{-1}\mathbf{a+v}}= \sum_{\mathbf{u}\in U}\sum_{\mathbf{a}\in S_{\mathbf{u}}}f_{\mathbf{a}} \mathbf{X}_l^{\mathbf{v_a}}\mathbf{Z}^{\mathbf{u_a}}= \sum_{\mathbf{u}\in U}F_{\mathbf{u}} \mathbf{Z}^{\mathbf{u}}.$$ Therefore for all $\mathbf z\in\mathbb G_m^l$ we have $F_{V,\mathbf z}=\sum_{\mathbf{u}\in U} F_\mathbf{u}(\mathbf z)\mathbf{X}_{n-l}^\mathbf{u}.$ So the coefficients of $F_{V,\mathbf z}$ are given by $F_{\mathbf{u}}(\mathbf z)$, $\mathbf{u}\in U$ and hence we have $$\log|F_{V,\mathbf z}|_p =\log\left( \max\{ |F_{\mathbf{u}}(\mathbf z)|_p: \mathbf{u}\in U\}\right).$$ 
Let $\tau\in \mathrm{Gal}(L(\boldsymbol{\eta})/L)$. Then there exists $\sigma\in G$ which restricts to $\tau$. Note that we have ${0\neq |F(\sigma(\boldsymbol{\zeta}))|_p= |F_{V,\sigma(\boldsymbol{\eta})}(\sigma(\boldsymbol{\xi}))|_p}$. Hence $F_{V,\tau(\boldsymbol{\eta})}\neq 0,$ especially for every $\tau$ there exists $\mathbf{u}\in U$ such that $F_{\mathbf{u}}(\tau(\boldsymbol{\eta}))\neq 0$. So we can apply Lemma \ref{lem:specialComb} to $F_{\mathbf{u}}, \mathbf{u}\in U$ to find $\boldsymbol{\mu}\in\mu_\infty^{n-l}$ such that ${F_{\boldsymbol{\mu}}=\sum_{\mathbf{u}\in U} \boldsymbol{\mu}^{\mathbf{u}} F_{\mathbf{u}}}$ does not vanish on $\tau(\boldsymbol{\eta})$ for all $\tau\in \mathrm{Gal}(L(\boldsymbol{\eta})/L)$, and ${|F_{\boldsymbol{\mu}}|_p=\max\{|F_{\mathbf{u}}|_p: \mathbf{u}\in U\}}$.
By the ultrametric inequality we have \begin{equation}\label{eq:Fmu}\begin{aligned}\frac{1}{[L(\boldsymbol{\eta}):L]}\sum_{\tau \in \text{Gal}(L(\boldsymbol{\eta})/L)} \log |F_{V,\tau(\boldsymbol{\eta})}|_p&= \frac{1}{[L(\boldsymbol{\eta}):L]}\sum_{\tau \in \text{Gal}(L(\boldsymbol{\eta})/L)}\log\left(\max_{\mathbf{u}\in U}|F_{\mathbf{u}}(\tau(\boldsymbol{\eta}))|_p\right)\\
&\geq \frac{1}{[L(\boldsymbol{\eta}):L]}\sum_{\tau \in \text{Gal}(L(\boldsymbol{\eta})/L)}\log |F_{\boldsymbol{\mu}}(\tau(\boldsymbol{\eta}))|_p.\end{aligned}\end{equation} 
Recall that the case $n=1$ is already done and the number of variables in $F_{\boldsymbol{\mu}}$ is $l\leq n-1$. So we can do induction on $n$ and assume that our theorem already holds for all polynomials in strictly less than $n$ variables. In particular it holds for $F_{\boldsymbol{\mu}}$.
Since $\mathrm{Gal}(L(\boldsymbol{\eta})/L)$ is isomorphic to $H'=\mathrm{Gal}(\mathbb{Q}(\boldsymbol{\eta})/\mathbb{Q}(\boldsymbol{\eta})\cap L)$ by restriction it can be viewed as a subgroup of $\mathrm{Gal}(\mathbb{Q}(\boldsymbol{\eta})/\mathbb Q)$. Recall that $E$ is the order of $\boldsymbol\eta$. Then $$[\mathrm{Gal}(\mathbb Q(\boldsymbol\eta)/\mathbb Q):H']=[\mathbb{Q}(\boldsymbol{\eta})\cap L:\mathbb{Q}]\leq [L:\mathbb{Q}]=[\mathrm{Gal}(\mathbb Q(\boldsymbol\zeta)/\mathbb Q):G].$$ The fixed field of $H'$ is contained in $L$ which itself is contained in the rationals adjoint by a root of unity of order $f_G$. So by Lemma \ref{lem:conductor} we have $f_{H'}\leq f_G$. The degree of $F_{\boldsymbol{\mu}}$ is bounded by $\deg(F_V)\ll_n \delta(\boldsymbol{\zeta})^{2\epsilon^{n-l}n}\deg(F)$. We have $F_{\boldsymbol{\mu}}=\sum_{\mathbf{u}\in U} \boldsymbol{\mu}^{\mathbf{u}} F_{\mathbf{u}}=F_{V}(\mathbf{X}_l,\boldsymbol{\mu})$. To show $\inf(F)\leq\inf(F_{\boldsymbol{\mu}})$ we let $\boldsymbol{\rho}\in\mu_\infty^l$. Then $|F_{\boldsymbol{\mu}}(\boldsymbol{\rho})|_p=|F_V(\boldsymbol{\rho,\mu})|_p=\left|F\left((\boldsymbol{\rho,\mu})^{V^{-1}}\right)(\boldsymbol{\rho,\mu})^{\mathbf{v}}\right|_p=\left|F\left((\boldsymbol{\rho,\mu})^{V^{-1}}\right)\right|_p$. This implies the inequalities between the infimas since $(\boldsymbol{\rho,\mu})^{V^{-1}}$ is a torsion point in $\mathbb{G}_m^n$.
By the ultrametric inequality we have $|F_{\boldsymbol{\mu}}|_p\leq |F|_p$, so we find $c(F_{\boldsymbol{\mu}})\leq c(F)$. 

Next we prove $\delta(\boldsymbol{\eta})\gg_n \delta(\boldsymbol{\zeta})^{1/2}$. Let $\mathbf{w}\in \mathbb{Z}^l\setminus \{{0}\}$ be such that $\boldsymbol{\eta}^{\mathbf{w}}=1$ and $|\mathbf{w}|=\delta(\boldsymbol{\eta})$. Then $$\boldsymbol{\zeta}^{V\begin{pmatrix}\mathbf{w} \\ \mathbf{0}\end{pmatrix}}=1\text{ and therefore }\delta(\boldsymbol{\zeta})\leq \left|V\begin{pmatrix}\mathbf{w} \\ \mathbf{0}\end{pmatrix}\right|\ll_n |V||\mathbf{w}|\ll_n\delta(\boldsymbol{\zeta})^{2\epsilon^{d-l}}\delta(\boldsymbol{\eta}).$$ Since $1-2\epsilon^{n-l}\geq 1/2$ the claimed inequality follows. To apply this theorem by induction on $n$, we have to check that $\delta(\boldsymbol{\eta})\geq C(l)\deg(F_{\boldsymbol{\mu}})^{C(l)}$. Since $\delta(\boldsymbol{\eta})\gg_n \delta(\boldsymbol{\zeta})^{1/2}$ and  $\deg(F_{\boldsymbol{\mu}})\ll_n \delta(\boldsymbol{\zeta})^{2\epsilon^{n-l}n}\deg(F)$ it is enough to check that $$\delta(\boldsymbol{\zeta})^{1-4\epsilon^{n-l}nC(l)}\gg_n C(l)^2\deg(F)^{2C(l)}.$$ If $C(n)$ is large enough in terms of $C(l)$ and $\epsilon$ small in terms of $n$ and $C(l)$, for example such that $1-4\epsilon^{n-l}nC(l)\geq 1/2$, then the inequality above follows from the hypothesis. We get \begin{align*}\frac{1}{\#H'}\sum_{\tau\in H'} \log |F_{\boldsymbol{\mu}}&(\tau(\boldsymbol{\eta}))|_p = \log |F_{\boldsymbol{\mu}}|_p + O_{n}\left(\frac{\deg(F_{\boldsymbol{\mu}})[(\mathbb{Z}/E\mathbb{Z})^*:H']^2f_H'^{1/2}\left(1+{c}(F_{\boldsymbol{\mu}})\right)}{\delta(\boldsymbol\eta)^{\kappa(l)}}\right)\\
&= \log |F_{\boldsymbol{\mu}}|_p + O_{n}\left(\frac{\delta(\boldsymbol{\zeta})^{2\epsilon^{n-l}n}\deg(F)[(\mathbb{Z}/N\mathbb{Z})^*:G]^2f_G^{1/2}\left(1+c(F)\right)}{\delta(\boldsymbol{\zeta})^{\kappa(l)/2}}\right).\end{align*} We may shrink $\epsilon$ a last time and assume that $2\epsilon^{n-l}n\leq \kappa(l)/4$. 
By construction we have $|F_{\boldsymbol{\mu}}|_p=\max\{|F_{\mathbf{u}}|_p: \mathbf{u}\in U\}$. But $|F_{\mathbf{u}}|_p = \max\{|f_{\boldsymbol{\alpha}}|_p: \boldsymbol{\alpha}\in S_{\mathbf{u}}\}$, so $|F_{\boldsymbol{\mu}}|_p= \max\{|f_{\boldsymbol{\alpha}}|_p: \boldsymbol{\alpha}\in S\}=|F|_p.$ Hence with \eqref{eq:Fmu} we get
$$\frac{1}{[L(\boldsymbol{\eta}):L]}\sum_{\tau \in \text{Gal}(L(\boldsymbol{\eta})/L)} \log |F_{V,\tau(\boldsymbol{\eta})}|_p\geq\log|F|_p+O_{n}\left(\frac{\deg(F)[(\mathbb{Z}/N\mathbb{Z})^*:G]^2f_G^{1/2}\left(1+{c}(F)\right)}{\delta(\boldsymbol{\zeta})^{\kappa(l)/4}}\right).$$
 With \eqref{eq:bdFmu} and since  $\frac{1}{\#G}\sum_{\sigma \in G} \log |F(\sigma(\boldsymbol{\zeta}))|_p\le  \log |F|_p$ by the ultrametric inequality we get$$\frac{1}{\#G}\sum_{\sigma \in G} \log |F(\sigma(\boldsymbol{\zeta}))|_p= \log |F|_p+ O_{n}\left(\frac{\deg(F)[(\mathbb{Z}/N\mathbb{Z})^*:G]^2f_G^{1/2}\left(1+{c}(F)\right)}{\delta(\boldsymbol{\zeta})^{\kappa}}\right),$$ where $\kappa\leq \min\left\{\frac{\nu_{l+1}^n}{48n},\kappa(l)/4\right\}$. This finishes the induction and the proof.\end{proof}

\begin{proof}[Proof of Theorem \ref{thm:pSimple}] Suppose that $\boldsymbol\zeta\in\mathbb G_m^n$ has finite order and satisfies $F(\boldsymbol\zeta)=0$. Laurent proved the Manin-Mumford conjecture for $\mathbb G_m^n$ over $\mathbb C$ in \cite{Laurent}. Since $\mathbb C$ and $\mathbb C_p$ are isomorphic, the Manin-Mumford conjecture is also true in $\mathbb C_p$ and implies that $\delta(\boldsymbol\zeta)$ is bounded. Thus if $\delta(\boldsymbol\zeta)$ is sufficiently large in terms of $n$ and $F$ we have $F(\boldsymbol\zeta)\ne0$ and the same is true for all conjugates of $\boldsymbol\zeta$. We fix $G=\mathrm{Gal}(\mathbb Q(\boldsymbol\zeta/\mathbb Q)$ have $[\mathrm{Gal}(\mathbb Q(\boldsymbol\zeta)/\mathbb Q):G]=1$ and $f_G=1$. The theorem now follows from Theorem \ref{thm:padic}.\end{proof}

\section{Ih's Conjecture}
In the last part we prove a special case of Ih's Conjecture for polynomials all of whose conjugates are essentially atoral.
\subsection{$S$-integers, $S$-units and $p$-adic absolute values}
We introduce $p$-adic absolute values on number fields and define $S$-units for a finite set of primes $S$.
\begin{definition}Let $K\subseteq \mathbb{C}$ be a number field with ring of integers $\mathbb{Z}_K$. Let $\{0\}\neq \mathfrak{P}\subseteq \mathbb{Z}_K$ be a prime ideal. Let $x\in K\setminus\{0\}$. We denote by $\nu_{\mathfrak{P}}(x)$ the exponent of $\mathfrak{P}$ in the factorisation of $(x)$. There exists a unique prime $p\in\mathfrak{P}$. Let us define $e(\mathfrak{P})=\nu_\mathfrak{P}(p).$ The norm $N(\mathfrak{P})$ is defined as the cardinality of $\kappa(\mathfrak{P})=\mathbb{Z}_K/\mathfrak{P}$ which equals $p^{f(\mathfrak{P})}$ for some natural number $f(\mathfrak{P})$. Then let $d_\mathfrak{P}=e(\mathfrak{P})f(\mathfrak{P})$.
Finally we define $$|x|_{\mathfrak{P}}=\begin{cases}p^{-\nu_\mathfrak{P}(x)/e(\mathfrak{P})} & \text{ if }x\in K\setminus\{0\} \\
0 &\text{ if } x=0\end{cases}.$$
By $M(K)$ we denote the set of places of $K$; the infinite ones by $M^\infty(K)$ and the finite ones by $M^0(K)$. Let $S\subseteq\mathbb{P}$ be a finite set of rational primes. Then we define $S_K^0=\{\nu\in M^0(K): \nu | p\text{ for some } p\in S\}$ and $S_K=M^\infty(K)\cup S_K^0\subseteq M(K),$ a finite set of places of $K$, containing the archimedian ones.
\end{definition}

\begin{lemma}\label{lem:absInd}Let $S\subseteq\mathbb{P}$ be a finite set of primes, $L/K$ an extension of number fields and $P\in K[X_1,\dots,X_n].$ Then $$\frac{1}{[K:\mathbb{Q}]}\sum_{\mathfrak{p}\in S_K^0}d_{\mathfrak{p}}\log |P|_{\mathfrak{p}}=\frac{1}{[L:\mathbb{Q}]}\sum_{\mathfrak{P}\in S_L^0}d_{\mathfrak{P}}\log|P|_{\mathfrak{P}}.$$\end{lemma}
\begin{proof}Let $\mathfrak p\in M_K^0, \mathfrak P\in M_L^0$ above $\mathfrak p$ and $x\in K\setminus \{0\}$. Using Satz III.1.2.v in \cite{Neukirch} we find that $$|x|_{\mathfrak P}^{d_\mathfrak P}=|N_{L_{\mathfrak P}/K_{\mathfrak p}}(x)|_{\mathfrak p}^{d_{\mathfrak p}} \text{ and hence }|N_{L_{\mathfrak P}/K_{\mathfrak p}}(x)|_{\mathfrak p}^{1/[L:K]}=|x|_{\mathfrak P}^{\frac{d_\mathfrak P}{d_\mathfrak p[L:K]}}.$$ We can now apply Lemma 1.3.7 in \cite{BombGub} and after multiplying by $\frac{d_{\mathfrak p}}{[k:\mathbb Q]}$ we get \begin{equation}\label{eq:LK}\frac{1}{[L:\mathbb Q]}\sum_{\mathfrak P|\mathfrak p}d_\mathfrak P \log |x|_{\mathfrak P}=\frac{d_{\mathfrak p}}{[K:\mathbb Q]}\log |x|_{\mathfrak p}.\end{equation} Note that we have $S_L^0=\bigcup_{\mathfrak{p}\in S_K^0} \{\mathfrak{P}\in M_L^0: \mathfrak{P}|\mathfrak{p}\}$ and the union is disjoint. Since $|\cdot|_{\mathfrak P}$ extends $|\cdot|_{\mathfrak p}$ we have $|x|_{\mathfrak P}<|y|_{\mathfrak P}$ if and only if $|x|_{\mathfrak p}<|y|_{\mathfrak p}$ for all $x,y\in K$. Thus equation \eqref{eq:LK} is also true for the Gauss norm of $P$. Summing over $\mathfrak p\in S^0_K$ we find $$\frac{1}{[K:\mathbb{Q}]}\sum_{\mathfrak{p}\in S_K^0}d_{\mathfrak{p}}\log |P|_{\mathfrak{p}}=\sum_{\mathfrak{p}\in S_K^0}\frac{1}{[L:\mathbb{Q}]}\sum_{\mathfrak{P}|\mathfrak{p}}d_{\mathfrak{P}}\log|P|_{\mathfrak{P}}=\frac{1}{[L:\mathbb{Q}]}\sum_{\mathfrak{P}\in S_L^0}d_{\mathfrak{P}}\log|P|_{\mathfrak{P}}.\qedhere$$\end{proof}

\begin{definition}Let $K$ be a number field and $S$ a finite set of places of $K$ containing the archimedian ones. Then the ring of $S$-integers is defined as $$\mathbb{Z}_{K,S}=\{x\in K: |x|_\nu\leq 1 \text{ for all }\nu\in M_K\setminus S\}.$$\end{definition}

\begin{remark}It is not hard to prove that the units of $\mathbb Z_{K,S}$ are given by $$(\mathbb{Z}_{K,S})^*=\{x\in K: |x|_\nu = 1 \text{ for all }\nu\in M_K\setminus S\}.$$\end{remark}

\begin{lemma}\label{lem:DefSUnit}Let $S$ be a set of primes and $x\in \overline{\mathbb{Q}}$. Then there are equivalent\begin{enumerate}
\item There exists a number field $K$ such that $x\in(\mathbb{Z}_{K,S_K})^*.$
\item For each number field $K$ containing $x$ we have $x\in(\mathbb{Z}_{K,S_K})^*.$\end{enumerate}\end{lemma}

\begin{proof}One direction is obvious. So let us assume that there exists a number field $K$ such that $x\in(\mathbb{Z}_{K,S_K})^*.$ Then $K$ contains $\mathbb{Q}(x)$. Let $\mathfrak{p}\in M({\mathbb{Q}(x)})\setminus S_{\mathbb{Q}(x)}$ and $\mathfrak{P}\in M(K)$ a prime ideal which lies above $\mathfrak{p}$. Then the prime in $\mathfrak{P}$ is not contained in $S$ and therefore $\mathfrak{P}\in M({K})\setminus S_{K}$. We have $|x|_{\mathfrak{p}}=|x|_{\mathfrak{P}}=1$ and therefore $x\in(\mathbb{Z}_{\mathbb{Q}(x),S_{\mathbb{Q}(x)}})^*.$ Now let $L$ be some number field which contains $x$ and $\mathfrak{P}\in M(L)\setminus S_L$. Let $\mathfrak{p}\in M({\mathbb{Q}(x)})$ be such that $\mathfrak{P}|\mathfrak{p}$. Then $\mathfrak{p}\in  M({\mathbb{Q}(x)})\setminus S_{\mathbb{Q}(x)}$ and therefore $|x|_{\mathfrak{P}}=|x|_\mathfrak{p}=1$, which implies that $x\in(\mathbb{Z}_{L,S_L})^*$.\end{proof}

\begin{definition}Let $S$ be a finite set of primes. If $x\in\overline{\mathbb{Q}}$ satisfies the conditions in Lemma \ref{lem:DefSUnit}, we call it $S$-unit.\end{definition}

\begin{lemma}\label{lem:conjSunit}Let $S$ be a finite set of primes and let $x$ be in an algebraic closure of $\mathbb Q$. If $x$ is an $S$-unit so are its conjugates.\end{lemma}
\begin{proof}Let $K/\mathbb Q$ be a Galois extension such that $x\in K$. Let $\nu\in M_K\setminus S_K$ be above $p$ and $\sigma\in\mathrm{Gal}(K/\mathbb Q)$. Then the map $z\mapsto |\sigma(z)|_{\nu}$ is a $p$-adic absolute value of $K$. Thus, there exists $\nu'|p$ in $M_K^0$ such that $|\sigma(z)|_{\nu}=|z|_{\nu'}$ for all $z\in K$. Since $\nu'$ is above $p$ outside $S$ and $x$ is an $S$-unit, we have $|\sigma(x)|_{\nu} = |x|_{\nu'} = 1$ and hence, $\sigma(x)$ is also an $S$-unit.\end{proof}

\subsection{Archimedian part}
We deduce from the main result in \cite{Atoral} an equidistribution theorem in the form we will use later.

\begin{lemma}\label{lem:arch}Let $K\subseteq\mathbb{C}$ be a number field. Let $P\in K[X_1,\dots,X_n]$ be a polynomial and $Q=\prod_{\tau:K\to\mathbb{C}}\tau(P)$. Assume that $\tau(P)$ is essentially atoral for all $\tau\in\mathrm{Gal}(K/\mathbb{Q})$. Then there exists $\kappa>0$ with the following property. Let $\boldsymbol{\zeta}\in\mathbb{G}_m^n$ be of finite order and $\delta(\bzeta)$ large in terms of $P$. Then $P(\bzeta)\ne 0$ and $$\frac{1}{[K(\boldsymbol{\zeta}):\mathbb{Q}]}\sum_{\sigma:K(\boldsymbol{\zeta})\to \mathbb{C}} \log|\sigma(P(\boldsymbol{\zeta}))| = \frac{m(Q)}{[K:\mathbb{Q}]} + O_{K,P}(\delta(\boldsymbol\zeta)^{-\kappa})\text{ as }\delta(\boldsymbol{\zeta})\to \infty,$$ where the implied constant does not depend on $\boldsymbol{\zeta}$.\end{lemma}

\begin{proof}Suppose that $\boldsymbol{\zeta}\in\mathbb{G}_m^n$ has order $N$. Let $G= \mathrm{Gal}(\mathbb{Q}(\boldsymbol\zeta)/K\cap\mathbb{Q}(\boldsymbol{\zeta}))$. Then we can bound $[\mathrm{Gal}(\mathbb{Q}(\boldsymbol\zeta)/\mathbb{Q}):G]=[K\cap\mathbb{Q}(\boldsymbol{\zeta}):\mathbb{Q}]\leq[K:\mathbb{Q}]$. Since the conductor of $G$ depends only on its fixed field by Lemma \ref{lem:conductor}, which in this case is $K\cap\mathbb{Q}(\boldsymbol{\zeta})$, it is bounded from above solely in terms of $K$, since there are only finitely many subfields of $K$. Thus the nominator of the error term in Corollary 8.9 in \cite{Atoral} is bounded in terms of $K$, $\deg(P)$ and the absolute logarithmic Weil height $h(P)$ alone. Let $\tau:K\to\mathbb{C}$ and $\tilde{\tau}:K(\boldsymbol{\zeta})\to \mathbb{C}$ be an extension of $\tau$. An application of Corollary 8.9 to $\tau(P), G$ and $\tilde{\tau}(\boldsymbol{\zeta})$ yields the existence of $\kappa>0$ such that $$\frac{1}{\#G}\sum_{\sigma\in G}\log|\tau(P)(\tilde{\tau}(\boldsymbol{\zeta})^\sigma)|=m(\tau(P))+O_{K,P}(\delta(\boldsymbol\zeta)^{-\kappa})\text { as }\delta(\tilde{\tau}(\boldsymbol{\zeta}))=\delta(\boldsymbol{\zeta})\to \infty.$$ Since all conjugates of $P$ have the same degree and the same height, the error does not depend on $\tau$. Now we average over $\tau: K\to \mathbb{C}$. By the additivity of the Mahler measure, the right hand side equals $m(Q)/[K:\mathbb{Q}] + O_{K,P}(\delta(\boldsymbol\zeta)^{-\kappa})$. We claim that the left hand side is given by $$\frac{1}{[K(\boldsymbol{\zeta}):\mathbb{Q}]}\sum_{\sigma:K(\boldsymbol{\zeta})\to \mathbb{Q}}\log|\sigma(P(\boldsymbol{\zeta}))|.$$ By Theorem \ref{thm:lang} $\mathrm{Gal}(K(\boldsymbol{\zeta})/K)$ is isomorphic to $G$ and hence $[K:\mathbb{Q}] \#G=[K(\boldsymbol{\zeta}):\mathbb{Q}].$ For all $\tau: K\to \mathbb{C}$ let us fix an extension $\tilde{\tau}:K(\boldsymbol{\zeta})\to \mathbb{C}$. We claim that $(\tau,\sigma)\mapsto \tilde{\tau}\sigma$ is a bijection between $\{\tau: K\to \mathbb{C}\} \times \mathrm{Gal}(K(\boldsymbol{\zeta})/K)$: Surely the map is well defined. Assume that $\tilde{\tau}\sigma=\tilde{\tau}'\sigma'$. If we restrict to $K$ we find $\tau=\tau'$ and hence $\tilde{\tau}=\tilde{\tau}'$. Since $\tilde\tau$ is injective we also have $\sigma=\sigma'$. Thus the map is injective. The two sets have the same finite size $[K(\boldsymbol{\zeta}):K][K:\mathbb{Q}]=[K(\boldsymbol{\zeta}):\mathbb{Q}]$ and hence any injection is also a bijection. Therefore we have \begin{align*}\sum_{\tau: K\to \mathbb{C}}\sum_{\sigma\in G}\log|\tau(P)(\tilde{\tau}(\boldsymbol{\zeta})^\sigma)|&=\sum_{\tau:K\to \mathbb{C}}\sum_{\sigma\in \mathrm{Gal}(K(\boldsymbol{\zeta})/K)}\log|\tilde{\tau}\sigma(P(\boldsymbol{\zeta}))|\\
&=\sum_{\sigma: K(\boldsymbol{\zeta})\to \mathbb{C}} \log|\sigma(P(\boldsymbol{\zeta}))|\end{align*} using again that $\tilde{\tau}\sigma$ restricts to $\tau$ on $K$ and that $\tilde{\tau}\sigma(\boldsymbol\zeta)=\sigma\tilde{\tau}(\boldsymbol\zeta).$\end{proof}

\subsection{$p$-adic part}
We apply our main result to state an equidistribution result for $p$-adic absolute values.
\begin{lemma}\label{lem:embed}Let $K$ be a number field and $\nu\in M^0(K)$ a finite place of it which lies above the rational prime $p$. Then there is an embedding $\iota:K\to \overline{\mathbb{Q}_p}$ such that $|x|_\nu=|\iota(x)|_p$ for all $x\in K$.\end{lemma}
\begin{proof}Clearly $K/\mathbb{Q}$ is algebraic and $|\cdot|_\nu$ extends $|\cdot|_p$. So by Theorem II.(8.1) in \cite{Neukirch} there exists a $\mathbb{Q}$-emedding $\iota:K\to\overline{\mathbb{Q}_p}$ such that $|x|_\nu=|\iota(x)|_p$ for all $x\in K$.\end{proof}

\begin{definition}Let $K$ be a number field and $P\in K[X_1,\dots,X_n]$. We define $$C(K,P,p)=-\log\left(\frac{1}{\max_{\iota:K\to\overline{\mathbb{Q}_p}}|\iota(P)|_p}\min_{\iota:K\to\overline{\mathbb{Q}_p}}\inf(\iota(P))\right).$$\end{definition}

\begin{lemma}\label{lem:CP}Let $K\subseteq\overline{\mathbb{Q}}\subseteq\mathbb{C}$ be a number field and $P\in K[X_1,\dots, X_n]\setminus\{0\}$. Let $\iota:K \to \overline{\mathbb{Q}_p}$. Then we have ${c}(\iota(P))\leq C(K,P,p)$.\end{lemma}
\begin{proof}We have ${c}(\iota(P))=-\log\left(\frac{1}{|\iota(P)|_p}\inf(\iota(P))\right)\leq C(K,P,p)$.\end{proof}

\begin{lemma}\label{lem:numbField}Let $K\subseteq\overline{\mathbb{Q}}$ be a number field and $P\in K[X_1,\dots, X_n]\setminus\{0\}$ a polynomial. Then there exist constants $C=C(n,P)\geq 1$ and $\kappa=\kappa(n)>0$ with the following property: Let $\boldsymbol{\zeta}\in \mathbb{G}_m^n$ have order $N$ and let $G<\mathrm{Gal}(\mathbb{Q}(\boldsymbol{\zeta})/\mathbb{Q})$. Let $\nu\in M^0({K(\boldsymbol{\zeta})})$ be a finite place over $p$. If $\delta(\boldsymbol{\zeta})> C$ then $P(\sigma(\boldsymbol{\zeta}))\ne 0$ for all $\sigma\in G$ and $$\frac{1}{\#G}\sum_{\sigma\in G} \log |P(\sigma(\boldsymbol{\zeta}))|_\nu = \log |P|_\nu + O_{n}\left(\frac{\deg(P)[\mathrm{Gal}(\mathbb{Q}(\boldsymbol{\zeta})/\mathbb{Q}):G]^2f_G^{1/2}\left(1+C(K,P,p)\right)}{\delta(\boldsymbol{\zeta})^\kappa}\right)$$ as $\delta(\bzeta)\to\infty$. The implied constant does neither depend on $\nu$ nor on $p$.\end{lemma}
\begin{proof}Let $\boldsymbol{\zeta}\in\mathbb{G}_m^n$ be of finite order. By the Manin-Mumford-Conjecture for $\mathbb G_m^n$ \cite{Laurent} the set of $\delta(\boldsymbol{\zeta})$ such that $P(\boldsymbol{\zeta})=0$ is bounded by $M=M(P,n)$ say. Let $C(n),\kappa=\kappa(n)$ be the constants from Theorem \ref{thm:padic} and define $C=\max\{M,C(n)\deg(P)^{C(n)}\}$. Now let $\boldsymbol{\zeta}\in \mathbb{G}_m^n$ have order $N$ and satisfy $\delta(\boldsymbol{\zeta})>C$. Let $G<\mathrm{Gal}(\mathbb{Q}(\boldsymbol{\zeta})/\mathbb{Q})$ and $\nu\in M^0({K(\boldsymbol{\zeta})})$ be a finite place. By Lemma \ref{lem:embed} there exists an embedding $\iota:K(\boldsymbol{\zeta})\to \overline{\mathbb{Q}_p}$ such that $|x|_\nu=|\iota(x)|_p$ for all $x\in K(\boldsymbol{\zeta})$. Let $F=\iota(P)\in\overline{\mathbb{Q}_p}[X_1,\dots,X_n]\subseteq \mathbb{C}_p[X_1,\dots,X_n].$ The embedding $\iota$ induces an isomorphism between $\mathbb{Q}(\boldsymbol{\zeta})$ and $L=\iota(\mathbb{Q}(\boldsymbol{\zeta}))$. Let $H=\{\iota\circ\sigma\circ\iota^{-1}: \sigma\in G\}$. Especially for all $\tilde{\sigma}\in H$ there is $\sigma\in G$ such that we have $\tilde{\sigma}(\iota(x))=\iota(\sigma(x))$ for all $x\in\mathbb{Q}(\boldsymbol{\zeta})$. Thus for all $\tilde{\sigma}\in H$ we find $$F(\tilde{\sigma}(\iota(\boldsymbol{\zeta})))=\iota(P)(\iota(\sigma(\boldsymbol{\zeta})))=\iota(P(\sigma(\boldsymbol{\zeta})))\neq 0$$ because $\delta(\sigma(\boldsymbol{\zeta}))=\delta(\boldsymbol{\zeta})>C\geq M$. Then by Theorem \ref{thm:padic} applied to $\iota(\boldsymbol{\zeta})\in\mathbb{G}_m^n(\mathbb{C}_p),\; H$ and $F$ we have $$\frac{1}{\#H}\sum_{\tilde{\sigma}\in H} \log |F(\tilde{\sigma}(\iota(\boldsymbol{\zeta})))|_p = \log |F|_p + O_{n}\left(\frac{\deg(F)[\mathrm{Gal}(\mathbb{Q}(\iota(\boldsymbol{\zeta}))/\mathbb{Q}):H]^2f_H^{1/2}\left(1+c(F)\right)}{\delta(\iota(\boldsymbol{\zeta}))^\kappa}\right).$$ Since $K(\boldsymbol\zeta)/K$ is a Galois extension $P(\sigma(\boldsymbol{\zeta}))$ lies in $K(\boldsymbol{\zeta})$ and thus $|F(\tilde{\sigma}(\iota(\boldsymbol{\zeta})))|_p=|\iota(P)(\iota(\sigma(\boldsymbol{\zeta})))|_p=|\iota(P(\sigma(\boldsymbol{\zeta})))|_p=|P(\sigma(\boldsymbol{\zeta}))|_\nu$. We also have $|F|_p=|P|_\nu$. Clearly the index and the conductor of $G$ and $H$ are the same and $c(F)\leq C(K,P,p)$ by Lemma \ref{lem:CP}. Finally since $\iota(\boldsymbol{\zeta})$ satisfies the same multiplicative relations as $\boldsymbol{\zeta}$ we have $\delta(\iota(\boldsymbol{\zeta}))=\delta(\boldsymbol{\zeta})$. Hence we find $$\frac{1}{\#G}\sum_{\sigma\in G} \log |P(\sigma(\boldsymbol{\zeta}))|_\nu = \log |P|_\nu + O_{n}\left(\frac{\deg(P)[\mathrm{Gal}(\mathbb{Q}(\boldsymbol{\zeta})/\mathbb{Q}):G]^2f_G^{1/2}\left(1+C(K,P,p)\right)}{\delta(\boldsymbol{\zeta})^\kappa}\right).\qedhere$$\end{proof}

\begin{lemma}\label{lem:pad}Let $K\subseteq\mathbb{C}$ be a number field, $P\in K[X_1,\dots,X_n]\setminus\{0\}$ a polynomial and $Q=\prod_{\tau: K\to \mathbb{C}}\tau(P)$. Then there exists $\kappa>0$ with the following property. Let $\boldsymbol{\zeta}\in\mathbb{G}_m^n$ be of finite order, $L$ the Galois closure of $K$ and $\nu\in M^0(L(\boldsymbol\zeta))$ a place above $p$. If $\delta(\bzeta)$ is large in terms of $P$ we have $P(\bzeta)\ne 0$ and $$\frac{1}{[K(\boldsymbol{\zeta}):\mathbb{Q}]}\sum_{\sigma:K(\boldsymbol{\zeta})\to\mathbb{C}} \log|\sigma(P(\boldsymbol{\zeta}))|_\nu = \frac{\log |Q|_\nu}{[K:\mathbb{Q}]} + O_{K,P,p}(\delta(\boldsymbol\zeta)^{-\kappa})\text{ as }\delta(\boldsymbol{\zeta})\to \infty,$$ where the implied constant does not depend on $\boldsymbol{\zeta}$.\end{lemma}

\begin{proof}Let $G=\mathrm{Gal}(\mathbb{Q}(\boldsymbol\zeta)/K\cap\mathbb{Q}(\boldsymbol{\zeta}))$. Then $[\mathrm{Gal}(\mathbb{Q}(\boldsymbol\zeta)/\mathbb{Q}):G]=[K\cap\mathbb{Q}(\boldsymbol{\zeta}):\mathbb{Q}]\leq[K:\mathbb{Q}]$. Since the conductor of $G$ depends only on its fixed field by Lemma \ref{lem:conductor}, which in this case is $K\cap\mathbb{Q}(\boldsymbol{\zeta})$, it is bounded from above solely in terms of $K$, since there are only finitely many subfields of $K$. Thus the nominator of the error term in Lemma \ref{lem:numbField} is bounded from above in terms of $K, \deg(P)$ and $C(K,P,p)$. Let $\tau: K\to \mathbb{C}$ and $\tilde{\tau}: K(\boldsymbol{\zeta})\to \mathbb{Q}$ be an extension of $\tau$. An application of Lemma \ref{lem:numbField} to $\tau(P), G$ and $\tilde{\tau}(\boldsymbol{\zeta})$ yields the existence of $\kappa>0$ such that $$\frac{1}{\#G}\sum_{\sigma\in G}\log|\tau(P)(\sigma\tilde{\tau}(\boldsymbol{\zeta}))|_\nu=\log |\tau(P)|_\nu+O_{\tau(K),P}\left(\frac{C(\tau(K),\tau(P),p)}{\delta(\boldsymbol\zeta)^{\kappa}}\right)\text { as }\delta(\tilde{\tau}(\boldsymbol{\zeta}))=\delta(\boldsymbol{\zeta})\to \infty.$$ Now we average over $\tau: K\to \mathbb{C}$. By the multiplicity of the Gauss norm, the right hand side equals $\log |Q|_\nu/[K:\mathbb{Q}] + O_{K,P,p}\left({\delta(\boldsymbol\zeta)^{-\kappa}}\right)$. By the exactly same argument as in the proof of Lemma \ref{lem:arch} we find that the left hand side is given by $$\frac{1}{[K(\boldsymbol{\zeta}):\mathbb{Q}]}\sum_{\sigma: K(\boldsymbol{\zeta})\to \mathbb{C}}\log|\sigma(P(\boldsymbol{\zeta}))|_\nu.\qedhere$$\end{proof}

\begin{lemma}\label{lem:Sint}Let $K\subseteq \mathbb C$ be a number field and $S$ a finite set of primes. Let $P\in K[X_1^{\pm 1}, \dots, X_n^{\pm 1}]$ be such that $\{\delta(\bzeta): \bzeta\in \mu_{\infty}^n,\; P(\bzeta)\text{ is an }S\text{-unit}\}$ is unbounded. Then $|P|_\nu = 1$ for every $\nu\in M_K\setminus S_K$. In particular, the coefficients of $P$ are $S$-integers.\end{lemma}
\begin{proof}
We can assume that $P$ is a nonzero polynomial. Fix $\nu_0\in M_K\setminus S_K$. Let $\bzeta\in\mu_{\infty}^n$ be a torsion point such that $P(\bzeta)$ is an $S$-unit. Let $G=\mathrm{Gal}(\mathbb Q(\bzeta)/K\cap \mathbb Q(\bzeta)).$ We can bound the index by $[\mathrm{Gal}(\mathbb Q(\bzeta)/\mathbb Q):G]=[K\cap \mathbb Q(\bzeta): \mathbb Q]\le [K:\mathbb Q]$. The fixed field $K\cap \mathbb Q(\bzeta)$ is a subfield of $K$, so there are only finitely many possibilities and thus $f_G$ is bounded in terms of $K$ by Lemma \ref{lem:conductor}. By Theorem \ref{thm:lang} the group $H=\mathrm{Gal}(K(\bzeta)/K)$ is isomorphic to $G$ by restriction to $\mathbb Q(\bzeta)$. 
	
	Let $\nu|\nu_0$ be a place in $M_{K(\bzeta)}\setminus S_{K(\bzeta)}$. Since conjugates of $P(\bzeta)$ are $S$-units by Lemma \ref{lem:conjSunit}, we have $$\frac{1}{\#G}\sum_{\sigma\in G} \log |P(\sigma(\bzeta))|_\nu = \frac{1}{\#H}\sum_{\sigma\in H} \log |P(\sigma(\bzeta))|_\nu= \frac{1}{\#H}\sum_{\sigma\in H} \log |\sigma(P(\bzeta))|_\nu=0.$$
	
	Let us apply Lemma \ref{lem:numbField}. Since the numerator of the error term is bounded independently of $\bzeta$, we find that there is $\kappa=\kappa(n)>0$ such that $$0=\log|P|_{\nu_0} + O_{n,K,p,P}(\delta(\bzeta)^{-\kappa})$$ for all $\bzeta\in\mu_{\infty}^n$ such that $P(\bzeta)$ is an $S$-unit. As $\delta(\bzeta)$ is as large as we wish by hypothesis, we find that $|P|_{\nu_0}=1$. Since $\nu_0\in M_K\setminus S_K$ was arbitrary, we can conclude that the coefficients of $P$ are $S$-integral.\end{proof}

\subsection{Extended cyclotomic polynomials and proof of Theorem \ref{thm:Sunits}}
We use the equidisribution results and the product formula to find a relation between the Mahler measure and the $p$-adic absolute values of $Q=\prod_{\tau:K\to\mathbb C}\tau(P)$. Using a characterisation of Boyd \cite{Boyd} of all the polynomials with integer coefficients and Mahler measure zero, we see that in this case the zero locus of $P$ is a finite union of torsion cosets. Note that for all $Q\in\mathbb Q[X_1,\dots, X_n]\setminus\{0\}$ we have $|Q|_p=1$ for all but finitely many primes $p$, thus the sum $\sum_{p\in\mathbb P}\log |Q|_p$ is well defined.

\begin{lemma}\label{lem:eqPQp}Let $S$ be a finite set of primes, $K\subseteq\mathbb{C}$ a number field. Let $P\in K[X_1,\dots,X_n]\setminus\{0\}$ be a polynomial, $Q=\prod_{\tau: K\to \mathbb{C}}\tau(P)$. Assume that $\tau(P)$ is essentially atoral for all $\tau: K\to \mathbb C$.\\
If $\{\delta(\boldsymbol{\zeta}): \boldsymbol{\zeta}\in \mu_\infty^n,\; P(\boldsymbol\zeta)\text{ is an }S\text{-unit}\}$ is unbounded, then $m(Q)+\sum_{p\in \mathbb P} \log|Q|_p=0$.\end{lemma}
\begin{proof}Let $\boldsymbol\zeta\in\mu_\infty^n$ be a torsion point such that $u=P(\boldsymbol\zeta)\text{ is an }S\text{-unit}$. In particular we have $u\ne 0$. Let $L$ be the Galois closure of $K$. For all $\sigma\in\mathrm{Gal}(L(\boldsymbol\zeta)/\mathbb Q)$ the map $|\cdot|_\nu\mapsto |\sigma(\cdot)|_\nu$ is a bijection of $M^0_{L(\boldsymbol\zeta)}$ to itself, such that the image lies over the same rational prime $p$. Since $S_{L(\boldsymbol\zeta)}^0$ contains either no or all places over $p$ we have \begin{equation}\label{eq:forts}\sum_{\nu\in S^0_{L(\boldsymbol\zeta)}} d_\nu\log|\sigma(u)|_\nu= \sum_{\nu\in S^0_{L(\boldsymbol\zeta)}} d_\nu \log|u|_\nu.\end{equation}

Let $\tau:K\to\mathbb C$ and $\tilde{\tau}:L(\bzeta)\to \mathbb C$ an extension. Then by the same argument we have
\begin{equation}\label{eq:norms}
	\sum_{\nu\in S^0_{L(\boldsymbol\zeta)}} d_\nu\log|\tau(P)|_\nu= \sum_{\nu\in S^0_{L(\boldsymbol\zeta)}} d_\nu\log|\tilde\tau(P)|_\nu= \sum_{\nu\in S^0_{L(\boldsymbol\zeta)}} d_\nu \log|P|_\nu
\end{equation}

Therefore by the product formula, and since $u$ is an $S$-unit we have $$-\sum_{\sigma\in\mathrm{Gal}(L(\boldsymbol\zeta)/\mathbb Q)}\log |\sigma(u)|
\stackrel{PF}=\sum_{\nu\in S^0_{L(\boldsymbol\zeta)}}d_\nu\log |u|_\nu\stackrel{\eqref{eq:forts}}=\sum_{\nu\in S^0_{L(\boldsymbol\zeta)}}d_\nu \frac{1}{[L(\boldsymbol\zeta):\mathbb Q]}\sum_{\sigma\in\mathrm{Gal}(L(\boldsymbol\zeta)/\mathbb Q)}\log |\sigma(u)|_\nu.$$ 
We divide by $[L(\boldsymbol\zeta):\mathbb Q]$ and use Lemmas \ref{lem:arch} and \ref{lem:pad} and $\prod_{\tau:L\to\mathbb C}\tau(P)=Q^{[L:K]}$ to find $\kappa>0$ such that $$\frac{[L:K]m(Q)}{[L:\mathbb Q]}+\frac{1}{[L(\boldsymbol\zeta):\mathbb Q]}\sum_{\nu\in S^0_{L(\boldsymbol\zeta)}}d_\nu \log |P|_\nu =O_{L,P}(\delta(\boldsymbol\zeta)^{-\kappa}) +\frac{1}{[L(\boldsymbol\zeta):\mathbb Q]}\sum_{\nu\in S^0_{L(\boldsymbol\zeta)}}d_\nu O_{L,P,p}(\delta(\boldsymbol\zeta)^{-\kappa})$$ as $\delta(\boldsymbol\zeta)\to \infty$.
Since $\sum_{\substack{\nu\in M(L(\boldsymbol\zeta))\\ \nu| p}}d_\nu = [L(\boldsymbol\zeta):\mathbb Q]$ and by definition every $\nu\in S^0_{L(\bzeta)}$ lies above a prime $p\in S$, the right hand side is $O_{L,P,S}(\delta(\boldsymbol\zeta)^{-\kappa})$ as ${\delta(\boldsymbol\zeta)\to\infty}$. With \eqref{eq:norms} we find $$\sum_{\nu\in S^0_{L(\boldsymbol\zeta)}} d_\nu\log|P|_\nu= \frac{1}{[K:\mathbb Q]}\sum_{\tau:K\to \mathbb C}\sum_{\nu\in S^0_{L(\boldsymbol\zeta)}} d_\nu \log|\tau(P)|_\nu= \frac{1}{[K:\mathbb Q]}\sum_{\nu\in S^0_{L(\boldsymbol\zeta)}} d_\nu \log|Q|_\nu.$$
Since the polynomial $Q$ has rational coefficients, we find by Lemma \ref{lem:absInd} that $$\frac{1}{[L(\boldsymbol\zeta):\mathbb Q]}\sum_{\nu\in S^0_{L(\boldsymbol\zeta)}} d_\nu\log|P|_\nu=\frac{1}{[K:\mathbb Q]}\frac{1}{[L(\bzeta):\mathbb Q]}\sum_{\nu\in S^0_{L(\bzeta)}}\log |Q|_\nu=\frac{1}{[K:\mathbb Q]}\sum_{p\in S}\log |Q|_p.$$ Hence we have $$\frac{m(Q)}{[K:\mathbb Q]}+\frac{1}{[K:\mathbb Q]}\sum_{p\in S}\log |Q|_p=O_{L,P,S}(\delta(\boldsymbol\zeta)^{-\kappa})\text{ as }\delta(\boldsymbol\zeta)\to \infty.$$ Note that the left hand side does not depend on $\boldsymbol\zeta$. Since by assumption there exists $\boldsymbol\zeta\in\mu_\infty^n$ such that $\delta(\boldsymbol\zeta)$ is arbitrary large we must have $m(Q)+\sum_{p\in S}\log |Q|_p =0$. 

Let $p\in \mathbb P\setminus S.$ We want to see that $|Q|_p=1$. Let $\tau\in\mathrm{Gal}(L/\mathbb Q)$ be an automorphism. Let $\nu|p$ be a place of $L$. Then there exists a place $\nu'|p$ of $L$ such that $|\tau(\cdot)|_\nu=|\cdot|_{\nu'}$. By Lemma \ref{lem:Sint} we know that $|P|_{\nu'}=1$ and hence, we find $|\tau(P)|_\nu=1$. 

Since $\tau\in \mathrm{Gal}(L/\mathbb Q)$ was arbitrary, we find that $$|Q|_p^{[L:K]}=|Q|_\nu^{[L:K]} = \prod_{\tau:L\to \mathbb C}|\tau(P)|_\nu = 1$$ by the Gauss lemma. So we have $|Q|_p=1$ and hence $\sum_{p\in \mathbb P\setminus S} \log|Q|_p =0$, which together with $m(Q)+\sum_{p\in S}\log |Q|_p =0$ implies the lemma.\end{proof}

\begin{definition}Let $n\in\mathbb{N}$. We call $P\in\mathbb{Z}[X_1,\dots, X_n]$ an extended cyclotomic polynomial, if we can write $P=X_1^{b_1}\dots X_n^{b_n}\phi_m(X_1^{a_1}\dots X_n^{a_n})$, where $\phi_m$ is the $m$-th irreducible cyclotomic polynomial, the $a_1,\dots, a_n\in\mathbb{Z}$ are coprime and $b_i=\max\{0, -a_i\deg\phi_m\}$ are minimal such that $P$ is a polynomial.
Then we define $K_n$ as the set of polynomials which are products of $\pm X_1^{a_1}\dots X_n^{a_n}, a_i\in\mathbb{N}_0$ and extended cyclotomic polynomials.\end{definition}

The following lemma requires a result of Boyd \cite{Boyd}.

\begin{lemma}\label{lem:torsionCosets}Let $Q\in\mathbb{Q}[X_1,\dots,X_n]\setminus\{0\}$ be a polynomial such that $m(Q)+\sum_{p\in \mathbb{P}}\log |Q|_p =0$. Then the zero set of $Q$ in $\mathbb{G}_m^n$ is a finite union of torsion cosets.\end{lemma}
\begin{proof}Let $c$ denote the content of $Q$ and $Q'=Q/c\in\mathbb{Z}[X]$. Then we have $|Q'|_p=1$ for all $p\in\mathbb{P}$. The Mahler measure of $c\in\mathbb{Q}\setminus\{0\}$ equals $\log|c|$. Then using additivity of the Mahler measure and the product formula we find by hypothesis that $$0=m(Q)+\sum_{p\in \mathbb{P}}\log |Q|_p=\log|c|+m(Q')+\sum_{p\in \mathbb{P}}\log |c|_p+\sum_{p\in \mathbb{P}}\log |Q'|_p=m(Q').$$ Thus by Theorem 1 in \cite{Boyd} we have $Q'\in K_n$. Let $a_1,\dots,a_n\in\mathbb{Z}$ be coprime. Then the zero set of $\phi_m(X^{a_1} \dots X^{a_n})$ in $\mathbb{G}_m^n$ is given by $$\{(x_1,\dots,x_n)\in\mathbb{G}_m^n: x_1^{a_1}\dots x_n^{a_n}=\zeta, \mathrm{ord}(\zeta)=m\} = \bigcup_{\mathrm{ord}(\zeta)=m}\{\mathbf x\in\mathbb{G}_m^n:{\mathbf{x^a}}=\zeta\},$$ a finite union of torsion cosets. Hence the zero set in $\mathbb{G}_m^n$ of any polynomial in $K_n$ is a finite union of torsion cosets. In particular this holds for the zero set of $Q'\in K_n$ which equals the vanishing locus of $Q$ in $\mathbb{G}_m^n$.\end{proof}

\begin{lemma}\label{lem:QcosPcos}Let $K\subseteq\mathbb{C}$ be a number field, $P\in K[X_1,\dots,X_n]$ and $Q=\prod_{\tau\in\mathrm{Gal}(K/\mathbb{Q})}\tau(P)$. If the zero locus of $Q$ in $\mathbb{G}_m^n$ is a finite union of torsion cosets the same is true for $P$.\end{lemma}
\begin{proof}Let $V\subseteq \mathbb{G}_m^n$ be an irreducible component of the zero locus of $P$ in $\mathbb{G}_m^n$. Since it is irreducible it is contained in an irreducible component of the vanishing locus of $Q$. By hypothesis such an irreducible component is a torsion coset. Since all components have dimension $n-1$ they must be equal and therefore also all irreducible components of the zero locus of $P$ are torsion cosets.\end{proof}

\begin{proof}[Proof of Theorem \ref{thm:Sunits}]
We may assume that $P$ is a polynomial. We assume that the set $\{\delta(\boldsymbol{\zeta}):{\boldsymbol{\zeta}\in \mu_\infty^n},\; P(\boldsymbol{\zeta})\text{ is an }S\text{-unit}\}$ is unbounded. We have to prove that the vanishing locus of $P$ in $\mathbb{G}_{m}^n$ must be a finite union of torsion cosets and that $P$ has $S$-integral coefficients. The latter holds by Lemma \ref{lem:Sint}.

Let $Q=\prod_{\tau:K\to\mathbb{C}}{\tau(P)}\in \mathbb Q[X_1,\dots, X_n]$. By Lemma \ref{lem:eqPQp} we find ${m(Q)+ \sum_{p\in \mathbb P} \log |Q|_p =0}.$ We can apply Lemma \ref{lem:torsionCosets} and conclude that the vanishing locus of $Q$ in $\mathbb{G}_{m}^n$ is a finite union of torsion cosets and so is the zero set of $P$ in $\mathbb{G}_{m}^n$ by Lemma \ref{lem:QcosPcos}.\end{proof}

\bibliographystyle{alpha}
\bibliography{padicEquidistr}
\end{document}